\documentclass[a4paper,12pt,draft,reqno]{amsart}

\usepackage{amssymb, amsfonts, amscd,amsfonts, amsmath, amsthm}
\usepackage{enumerate}
\usepackage{mathrsfs}

\usepackage{graphics}
\usepackage{dsfont}
\usepackage{mathrsfs}

\usepackage[all]{xy}

\DeclareMathSymbol{\twoheadrightarrow}  {\mathrel}{AMSa}{"10}

{\catcode`\@=11
\gdef\n@te#1#2{\leavevmode\vadjust{%
 {\setbox\z@\hbox to\z@{\strut#1}%
  \setbox\z@\hbox{\raise\dp\strutbox\box\z@}\ht\z@=\z@\dp\z@=\z@%
  #2\box\z@}}}
\gdef\leftnote#1{\n@te{\hss#1\quad}{}}
\gdef\rightnote#1{\n@te{\quad\kern-\leftskip#1\hss}{\moveright\hsize}}
\gdef\?{\FN@\qumark}
\gdef\qumark{\ifx\next"\DN@"##1"{\leftnote{\rm##1}}\else
 \DN@{\leftnote{\rm??}}\fi{\rm??}\next@}}

\DeclareFontFamily{OT1}{wncyr}{\hyphenchar\font45 }
\DeclareFontShape{OT1}{wncyr}{m}{n}{%
   <5> <6> <7> <8> <9> gen * wncyr
   <10> <10.95> <12> <14.4> <17.28> <20.74>  <24.88>wncyr10}{}
\DeclareFontShape{OT1}{wncyr}{m}{it}{%
   <5> <6> <7> <8> <9> gen * wncyi
   <10> <10.95> <12> <14.4> <17.28> <20.74> <24.88> wncyi10}{}
\DeclareFontShape{OT1}{wncyr}{m}{sc}{%
   <5> <6> <7> <8> <9> <10> <10.95> <12> <14.4>
   <17.28> <20.74> <24.88>wncysc10}{}
\DeclareFontShape{OT1}{wncyr}{b}{n}{%
   <5> <6> <7> <8> <9> gen * wncyb
   <10> <10.95> <12> <14.4> <17.28> <20.74> <24.88>wncyb10}{}
\input cyracc.def

\DeclareMathSizes{9}{9}{7}{5} % This only is different.

\DeclareMathSymbol{\twoheadrightarrow} {\mathrel}{AMSa}{"10}

\theoremstyle{plain}

\newtheorem{theorem}{Theorem}[section]

\newtheorem{lemma}[theorem]{Lemma}%%[section]
\newtheorem{remark}[theorem]{Remark}%%[section]

\newtheorem{corollary}[theorem]{Corollary}%%[section]

\newtheorem{example}[theorem]{Example}

\theoremstyle{definition}

\newtheorem{definition}[theorem]{Definition}%%[section]

\newtheorem{nothing*}[theorem]{}
\newtheorem{subnothing*}[sub]{}

\newtheorem*{examples}{Examples}

\theoremstyle{remark}

\def\Aut {{\rm Aut\,}}
\def\Autn{{\rm Aut}(\mathbf A\!^{n})}

\def\An{{\mathbf A}\!^n}
\def\GL{{\rm GL}}

\def\Cr{{\rm Cr}}
\def\Aff{{\rm Aff}}
\def\Am{{\mathbf A}\!^m}

\def\Cn{{\rm Cr}_n}

\def\Cl{{\rm Cr}_1}
\def\Cp{{\rm Cr}_2}
\def\C3{{\rm Cr}_3}

\def\bAn{{\mathbf A}\!^{n}}
\def\bAm{{\mathbf A}\!^{m}}
\def\bA2{{\mathbf A}\!^2}
\def\bAl{{\mathbf A}\!^1}

\newcommand{\dss}{\hskip -2mm\rotatebox{68}{\raisebox{-1.8\height}{\mbox{\normalsize -\hskip .1mm-\hskip .1mm-}}}\hskip -.6mm}

\numberwithin{equation}{section}

\begin{document}
%%\renewcommand{\baselinestretch}{2}

%%\

%%\vskip -6mm

\title[Three plots about the Cremona groups]{Three plots about the Cremona groups}

\author[Vladimir  L. Popov]{Vladimir  L. Popov}%%${}^*$}
\address{Steklov Mathematical Institute,
Russian Academy of Sciences, Gubkina 8, Mos\-cow\\
119991, Russia
%%\ \vskip -6mm \
}

%%\address{National Research University\\ Higher School of Economics\\ %%Myasnitskaya
%%20\\ Mos\-cow 101000,\;Russia\
%%\vskip 3mm
%%}

\email{popovvl@mi-ras.ru}

%%\thanks{  }

%%\date{August 20, 2009}

%%\subjclass[2000]{14M99, 14L30, 14R20, 14L24, 17B45}

\dedicatory{To the memory of V.\;A.\;Iskovskikh}

\maketitle

%%\renewcommand{\baselinestretch}{1.0}

 %%%%%%%%%%%%%%%%%%%%%%

\begin{abstract}
The first group of results of this paper concerns the compressibility of finite subgroups of the Cremona groups. The second concerns the embeddability of other groups in the Cremona groups and, conversely, the Cremona groups in other groups. The third concerns the connectedness of the Cremona groups.
\end{abstract}

\section{Introduction}

\subsection{}
The Cremona group $\Cn(k)$ of rank $n$
over the field $k$ is the group of $k$-automorphisms of the
field $k(x_1, \ldots, x_n)$ of rational functions over $k$
in the variables $x_1, \ldots, x_n $.
It admits a geometric
interpretation: if the field $k (x_1, \ldots, x_n)$
is identified by means of a $k$-isomorphism
 with the field $k(X)$ of an irreducible algebraic
 variety $X$ defined and rational over $k$, then each element
$\sigma$ of the group ${\rm Bir}_k(X)$
of all $k$-birational self-maps $X\dashrightarrow X $
determines the element $\sigma^* \in \Cn(k)$,
\begin{equation} \label{ast}
\sigma^*(f): = f \circ \sigma, \; \; f \in k(X),
\end {equation}
and the mapping ${\rm Bir}_k (X) \to \Cn(k)$,
$\sigma \mapsto (\sigma^{- 1})^*$, is an isomorphism of groups.
For this reason, the group ${\rm Bir}_k(X) $ is
called the Cremona group as well and denoted by $\Cn(k)$.\;Which interpretation of $\Cn(k)$ is meant\,---\,algebraic
or geometric\,---\, is usually
clear from context.\;The naturally defined
concept of a morphism of an al\-geb\-raic variety
into the Cremona group (or the concept of an algebraic
family of elements of the Cremona group) allows one to endow it
with the Zariski topology  \cite[1.6]{Se10}. Besides
this property, there is a number of others which permit to
speak about the far-reaching analogies between the
Cremona groups and affine algebraic
groups, see \cite{Po13_1}, \cite{Po13_2},
\cite {Po14_2},  \cite {Po17}.

The Cremona groups
are classical objects of research,
intensity of which in recent years
increased significantly and led to essential advances
in understanding the structure of these groups.\;Among the most impressive is {\it tour de force}
\cite{DI09} by
I.\;V. Dolgachev and V.\;A.\;Iskovskikh
on the classification of finite subgroups of $\Cp(\mathbb C)$.

\subsection{} In this paper,
three aspects of the structure of the Cremona groups are explored.

 The topic of the (longest) Section \ref{compr}
is the comparison of different
finite subgroups of the Cremona group $\Cn(k)$,
where $k$ is an algebraically closed
field of characteristic zero.\;So far, in the studies of these subgroups,
including that in \cite{DI09}, they were all considered
on an equal footing. However, in reality  it is necessary to consider
some of them as ``not basic'',
since they are obtained from others by a standard
``base change'' construction \cite[3.4]{Po14_1}.\;This leads to the problem, formulated in \cite [3.4]{Po14_1}, \cite[Quest.\;1]{Po16},
of finding those subgroups in the classification
lists that are obtained by such a nontrivial
change, or, in another terminology,
are nontrivialy ``compressible'' (see definitions in Subsection \ref{deff}).

Developing this topic, in Section \ref{compr} we prove a series of statements about such subgroups. Some of them are of a general nature, while some concern
cases $n=1$ and $2$. For example, we obtain the following result (Theorem \ref{Ps}), which immediately implies the nontrivial self-compressibility of any finite subgroup $G$
in ${\rm Cr}_1$:\;for the cor\-responding
binary group $\widetilde G$ of linear transformations of the affine plane,
 we find an infinite increasing sequence of integers $d>0$ such that
 $\widetilde G$ admits a homogeneous polynomial self-compression of degree $d$,
 which descends to a nontrivial self-comp\-res\-sion of the group $G$. The proof
 allows us in principle to specify these self-compressions by explicit formulas.\;For $n=2$, we prove, for example, that if $G$  is a non-Abelian finite subgroup
of ${\rm GL}_2 (k) \subset
\Cp (k)$ that is not isomorphic to a dihedral group, then every
finite subgroup in $\Cp(k) $,
isomorphic to $G$ as an abstract group, is obtained from $G$ by a nontrivial base change
(Theorem \ref{nal}).\;Other statements on this subject, proved in Section \ref{compr}, see below in Theorems \ref{coli}--\ref{nal} and their Corollaries.

\subsection{} The subject of  Section \ref {2} is the embeddability of other groups in the Cremona groups and, conversely,  the embeddability of the Cremona groups in other groups. This theme originates from the question of J.-P. Serre \cite[\S6, 6.0] {Se09} on the existence of finite groups that are nonembeddable in ${\rm Cr}_3({\mathbf C})$.\;By now (September 2018)
 signi\-fi\-cant information is accumulated on it (including the affirmative answer to this question). The most essential contribution to its obtaining is related to the Jordan property (see Definition \ref{J} below) of the Cremona groups ${\rm Cr}_n (k) $, whose
 proof for any $n$ has been completed recently\footnote{ In \cite[Thm.\;1.8]{PS16}, it was given the conditional (modulo the so-called BAB conjecture) proof of the Jordan property of the group ${\rm Bir}_k (X)$ for any rationally connected algebraic $k$-variety $X$ in the case of ${\rm char}\,k = 0$ (and therefore, the conditional proof of the Jordan property of any Cremona group $\Cn(k)$).\;The BAB conjecture was then proved in
\cite[Thm.\;3.7]{Bi17}. This completed the proof of the Jordan property of the groups ${\rm Bir}_k(X) $.}\label {foot}.\;Although the statements about the group embeddings proved in Section \ref {2}
are also related to the Jordan property, which is
in the focus of attention already for a long time, in the published literature they did not occur to the author.

The fact that, for ${\rm char}\,k = 0$, every finite $p$-subgroup of $\Cn(k)$ is Abelian for sufficiently big $p$, immediately follows from the Jordan property of the Cremona groups (this was noted already in \cite [\S6, 6.1]{Se09}).\;Therefore,  every non-Abelian finite $p$-group (such exist for any $p$) is nonembeddable in $\Cn(k)$  for sufficiently big $p$.\;We prove
(Corollary \ref{sl}), for any Cremona group ${\rm Cr}_n(k)$ with ${\rm char}\,k=0$, the existence of an integer  $b_{n, k}>0$ such that every product of groups $G_1\times\cdots\times G_s$, each of which contains a non-Abelian finite subgroup, is nonembeddable in the group ${\rm Cr}_n(k)$ if $s>b_{n, k}$.
%%
%%We prove (Theorem \ref{2121}) that for any group $G$ containing a non-Abelian finite %%subgroup and any Cremona group $ {\rm Cr}_n(k) $, where $ {\rm char}\,k = 0$, there %%exists an $s$ such that $ G^s$ is nonembeddable in ${\rm Cr}_n(k)$.
In particular,
for any (and not only for sufficiently big) prime integer $p$, there exists a non-Abelian finite $p$-group that is nonembeddable in ${\rm Cr}_n(k)$.

Considering $p$-subgroups delivers invariants, which
allow us to prove in some cases that one group is nonembeddable in another.\;Some applications are obtained on this way.

For example (special case of Corollary \ref{c1}),
we prove
%%of Theorem \ref{Rd})
that if
$k$ is an algebraically closed field of characteristic zero, and with each integer $d>0$
any abstract group $H_d $ from
 the following list is associated:

 \begin{enumerate} [\hskip 4.2mm \rm (a)]
 \item ${\rm Cr}_d(k)$,
 \item ${\rm Aut} ({\bf A}_k^{\hskip -.6mm d})$,
 \item a connected affine algebraic group over $k$ with maximal tori
 of dimension $d$,
\item a connected real Lie group with maximal tori
 of dimension $d$,
 \end{enumerate}
  then the group $H_n$ is nonembeddable in $H_m$ if $n>m$. In particular,
  the groups $H_n$ and $H_m$ for $n\neq m$ are not isomorphic.
%%(Corollary \ref{c1} of Theorems \ref{Rd}, \ref{al}).
For instance,
${\rm Cr}_n(k)$ is embeddable in  ${\rm Cr}_m(k)$
if and only if  $n\leqslant m$; in particular, ${\rm Cr}_n$ and
${\rm Cr}_m$ are isomorphic if and only if  $n=m$ (this was previously proven in \cite[Thm.\;B]{Ca14}, \cite[Rem.\;1.11]{PS16}).

Another example (Theorem
\ref{tp111}): we prove that if $M$ is a compact connected $n$-dimensional topological manifold, and $B_M $ is the sum of its Betti numbers with respect to homology with coefficients in $\mathbf Z $, then for
$$d>\frac{\sqrt{n^2+4n(n+1)B_{M}}+n}{2}+{\rm log}_2B_{M},$$
the Cremona group $ {\rm Cr}_d (k) $ is nonembeddable in the homeomorphism group
 of the manifold $M$.

Concerning other statements on nonembeddable groups proved in Section \ref{2}, see  below Lemma \ref{Jsss}, Theorems \ref{Rd}, \ref{al}, \ref {t15} and their Corollaries.

\subsection {}
In Section \ref{1}, we return back to the question of J.-P. Serre
on the connectedness of the Cremona group $\Cn(k) $ in the Zariski topology
\cite[1.6]{Se10}.\;It was answered in the affirmative in \cite{BZ18}, where the linear connectedness (and therefore the connectedness) of the group $\Cn(k)$ is proved in the case of an infinite field $k$ (for an algebraically closed field $k$, this was proved earlier
in \cite{Bl10}).\;We give a short new proof for the case of an infinite field $k$.\;It is based on an argument, ideologically close to that of
Alexander, which he used in \cite{Al23} in proving the connectedness
of the homeomorphism group of the ball, and which was then adapted in \cite[Lem.\;4]{Sh82},
\cite [Thm.\,6] {Po14_2}, and \cite{Po17} to the proofs of connectedness of
the groups ${\rm Aut}(\bAn)$ and their affine-triangular subgroups, respectively.

\vskip 2mm

The author is grateful to J.-P.\;Serre, Ch.\;Urech,  and the referee for the comments.

\subsection{Notations and conventions}\

\vskip 1mm

$\overline k$ is a fixed algebraically closed field containing $k$.

$\Cn:=\Cn(\overline k)$, ${\rm Bir}(X):={\rm Bir}_{\overline k}(X)$, ${\rm Aut}(X):={\rm Aut}_{\overline k}(X)$.

$o=(0,\ldots, 0)\in\bAn$.

$\langle S\rangle$ is a linear span of a subset $S$ of a linear space over\;$k$.

${\rm Grass}(n, V)$ is the Grassmannian of all
$n$-dimensional linear subspaces of a finite-dimensional linear space $V$ over $k$.

${\mathbf P}(V):={\rm Grass}(1, V)$. We put ${\mathbf P}(\{0\})=\varnothing$ and $\dim (\varnothing)=-1$.

$L^{\oplus m}$ is the direct sum of
$m$ copies of a linear space $V$ over $k$
(for $m=0$, it is considered to be zero).

$G^{s}$ is the direct product of  $s$ copies of a group $G$.

``Variety'' means ``algebraic variety over $k$''.\;Its irreducibilitty
means geometric irreducibility, and points mean
$\overline k$-points.
The set of $k$-points of a variety $X$ is denoted by $X(k)$.

${\rm Dom}(\varphi)$ is the domain of definition of a rational map $\varphi$.

${\rm T}_{a, X}$ is the tangent space of a variety $X$ at a point $a$.

$d_a\varphi\colon {\rm T}_{a, X}\to {\rm T}_{\varphi(a), X}$ is the differential of a rational map $\varphi\colon X\dashrightarrow X$ at a point $a\in {\rm Dom}(\varphi)$.

$k[x_1,\ldots, x_n]_d$ is the space of all forms of degree $d$ in variables $x_1,\ldots, x_n$ and with coefficients in $k$.

$L_d\!:=\!L\cap k[x_1,\ldots, x_n]_d$ for any $k$-linear subspace $L$ in $k[x_1,\ldots, x_n]$.

$FH=\{fh\mid f\in F, h\in H\}$ for any nonempty sets $F, H\subseteq  k[x_1,\ldots, x_n]$.

The variables
$x_1,\ldots, x_n$ in the definition of the Cremona group are assumed to be the standard coordinate functions on $\bAn$:
\begin{equation*}
x_i(a):=a_i,\quad a:=(a_1,\ldots, a_n)\in \bAn.
\end{equation*}

For any rational map
$\sigma\colon \bAn\dashrightarrow\bAn$, we use the notation
\begin{equation}\label{string}
\sigma=(\sigma_1,\ldots, \sigma_n)\colon \bAn\dashrightarrow\bAn,\quad\mbox{where $\sigma_i:=\sigma^*(x_i)$.}
\end{equation}
We call $\sigma$ {\it polynomial homogeneous map of degree} $d$,
if $\sigma_1,\ldots, \sigma_n\in k[x_1,\ldots, x_n]_d$.

In these notation, if for a rational map
$\tau\colon \bAn\dashrightarrow \bAn$ the composition $\nu:=\sigma\circ\tau$ is defined,
then it is described by the formula
\begin{equation}\label{composi}
\nu_i=
\tau^*(\sigma_i)\;\;\mbox{for all $i$},
\end{equation}
i.e.,\;$\nu_i$ is obtained from the rational function $\sigma_i$ in $x_1,\ldots, x_n$
by means of plugging in $\tau
_j$ in place of $x_j$ for every  $j$.

The map \eqref{string}  is called {\it affine} (respectively, {\it linear}),
if all nonzero functions $\sigma_i$ are polynomial in $x_1,\ldots, x_n$ of degree $\leqslant 1$ (respectively, are forms in $x_1,\ldots, x_n$ of degree 1). The set of all
invertible affine (respectively, linear) maps $\bAn\to \bAn$ is the subgroup
$\Aff_n$ (respectively, $\GL_n$) of $\Cn$.

\section{Compressing finite subgroups of the Cremona groups}\label{compr}

In this section,  $k=\overline k$ and ${\rm char}\,k=0$.

\subsection{\bf Terminology}\label{deff}

First, we fix the terminology.

Here, unless a special reservation is made, a {\it rational action} of a finite group $ G $ on an irreducible variety $X$ is understood as a {\it faithful} (that is, with trivial kernel)
action by birational self-maps of this variety.\;Specifying such an action is equi\-valent to specifying a group embedding $\varrho \colon G \hookrightarrow {\rm Bir} (X) $; therefore, hereinafter the very homomorphism $\varrho $ is called a rational action.\;The integer $\dim (X) $ is called the {\it dimension of the action} $\varrho $.\;We say that $\varrho (G)$ is the subgroup of ${\rm Bir} (X)$ defined by the action $\varrho $.\;If $ \varrho (G) \subseteq {\rm Aut}(X) $, then the action $\varrho $ is called {\it regular}.

 Any regular action $\rho$ of the group $G$ on an irreducible smooth complete variety $Y$ such that there is a $G$-equivariant birational isomorphism  $X\dashrightarrow Y$
 is called a {\it regularization of the action} $\varrho$;
 combining the results of  \cite[Thm.\;1]{Ro56} and \cite{BM97} shows that a regularization always exists.
If there is a regularization $\rho$ such that  $Y^G\neq \varnothing$, then we say that
$\varrho$ {\it has a fixed point}.

Consider two rational actions
$\varrho_i\colon G
\hookrightarrow
{\rm Bir}(X_i)$, $i=1, 2$. Let
$\pi^{}_{i}\colon X_i\dashrightarrow  X_i\dss G$, $i=1, 2$, be the corresponding rational quotients, see \cite[2.4]{PV94}.\;Assume that there is a
$G$-equivariant dominant rational map
$\varphi\colon X_1\dashrightarrow X_2$.\;Let
$\varphi^{}_{G}\colon X_1\dss G\dashrightarrow X_2\dss G$
be the dominant rational map induced by $\varphi$.
Then the following properties hold
(see, e.g.,\;\cite[2.6]{Re00_1}):

First, the commutative diagram
\begin{equation}\label{diag}
\begin{matrix}
\xymatrix@=10mm{X_1\ar@{-->}[r]^{\varphi}\ar@{-->}[d]_{\pi^{}_{1}}&X_2\ar@{-->}[d]^{\pi^{}_{2}}\\
X_1\dss G\ar@{-->}[r]^{\varphi^{}_{G}}& X_2\dss G
}
\end{matrix}
\end{equation}
\noindent
is cartesian, i.e.,
$\pi^{}_{1}$
is obtained from $\pi^{}_{2}$
by the base change $\varphi^{}_{G}$.\;In particular,  $X_1$
is birationally $G$-equivariantly isomorphic to the variety
\begin{equation}\label{close}
X_2\times _{X_2\dss G}(X_1\dss G):=\overline{\{(x, y)\in
{\rm Dom}(\pi_2)\times
{\rm Dom}(\varphi_G) \mid
\pi^{}_{2}(x)=\varphi^{}_{G}(y)\}}
\end{equation}
\noindent(the bar in  \eqref{close} means the closure in $X_2\times X_1\dss G$), on which $G$ acts through \;$X_2$.

Second, for every irreducible variety $Z$ and every dominant rational map
$\beta\colon Z\dashrightarrow X_2\dss G$ such that the variety $X_2\times_{X_2\dss G}Z$
is irreducible, the latter
inherits through $X_2$ a rational action of the group $G$ such that commutative diagram \eqref{diag} with $ X_1 = X_2 \times_ {X_2 \dss G} Z $, $ \varphi ^ {} _ {G}
=\beta $, and $\varphi={\rm pr}_1$ holds.

    It is said \cite{Re04} that $\varphi$  is a  {\it compression} of the action ${\varrho}_1$ into the action  ${\varrho}_2$ (or that ${\varrho}_2$ {\it is obtained by the compression} $\varphi$ from ${\varrho}_1$), and also \cite[3.4]{Po14_1} that
   ${\varrho}_1$ {\it is obtained by the base change} $\varphi_G$
    from ${\varrho}_2$.\;A compression that is not (or, respectively, is) a birational isomorphism is called {\it nontrivial} (respectively, {\it trivial}\,);
    in this case, we say that ${\varrho}_1$ {\it is obtained by a nontrivial
    $(${\rm respectively,} trivial\,$)$ base change} from ${\varrho}_2$.\;If for $\varrho_1$ there is $\varrho_2$, which is obtained from  $\varrho_1$ by a nontrivial compression, then we say that  $\varrho_1$ is {\it nontrivially compressible}, and otherwise,  that
    it is  {\it incompressible}.
   Similar terminology applies to groups: if $G_i\subseteq {\rm Bir}(X_i)$, $i=1,2$, are finite subgroups isomorphic to $G$, then we say that $G_1$ is {\it compressible} into $G_2$, if there are rational actions
   $\varrho_i\colon G
\hookrightarrow
{\rm Bir}(X_i)$, $i=1, 2$, such that $\varrho_i(G)=G_i$, $i=1, 2$, and $\varrho_2$ is obtained by a compression $\varphi$ from $\varrho_1$.\;If $\varrho_1,\varrho_2, \varphi$  can be chosen so that $\varphi$ is nontrivial, then $G_1$ is called {\it nontrivialy compressible} into $G_2$.\;If $G_1$ does not admit any nontrivial compression into any subgroup in
${\rm Bir}(X_2)$, then $G_1$ is called {\it incompres\-sible}.

In the case when $X_1=X_2$ and $\varrho_1=\varrho_2$, we are talking about
 {\it self-compressions} of the action and the group.\;In particular, if in this case there exists a nontrivial compression, then we say that $\varrho_1(G)$ is a
{\it nontrivially self-compressible} subgroup of ${\rm Bir}(X)$.

\subsection{Self-compressibility of finite subgroups in  \boldmath$\Cl$:\;reformulation}\label{intr}
First we consider the problem of nontrivial self-compressibility of finite subgroups
  in the Cremo\-na group $\Cl $ of rank 1. It can be reformulated as follows.

  We assume that $\bAl=\{(a_0:a_1)\in \mathbf P^1\mid a_0\neq 0\}$ and denote  the standard coordinate function  $x_1\in k[\bAl]$ by $z$. The elements of every finite subgroup $G$ of the Cremona group $\Cl$ are fractional-linear functions from the field
  $k(z)$ (considered as rational maps $\bAl\dashrightarrow \bAl$).\;The restriction to $\bAl$  defines a bijection between the set of self-compressions $\mathbf P^1\to \mathbf P^1$ of the group $G$ and the set of rational functions $f=f(z)\in k(z)$, which are solutions of the following system of functional equations:
 \begin{equation}\label{fueq}
f\Big(\frac{az+b}{cz+d}\Big)=\frac{af+b}{cf+d}\quad\mbox{for all}\;\; \frac{az+b}{cz+d}\in G.
\end{equation}

In this setting, the nontriviality of the self-compression defined by the rational func\-tion
$f$ is equivalent to the condition $\deg(f)\geqslant 2$.\;Note that in \eqref{fueq} instead of
all functions from the group $G$ it suffices to consider only the generators of this group.

Thus, the question of nontrivial compressibility of the group $G$  is equivalent to the question of the existence of a rational function $f$ of degree $\geqslant 2$ among the solutions of system  \eqref{fueq}.

\subsection{Compressibility of binary polyhedral groups:\;formulation of the re\-sult} The comprehensive answer to the above question can be obtained for any finite subgroup of the Cremona group $\Cl$:\;all of them are nontrivially compressible.
This asnwer is an immediate corollary of a
 more subtle result, which we obtain here.\;Namely,
 we prove that there exists
 infinitely many homogeneous polynomial self-compressions $\bf A\!^2\to \bf A\!^2$
 of any binary polyhedral group,
 descending to the nontrivial self-compressions
  $\bf P^1\to \bf P^1$ of the corresponding
  polyhedral group.\;The proof is effective and gives a way to explicitly specify these self-compressions by formulas
  (see Remark (c) in Subsection \ref{comment}).

  We now give the precise formulation of this result.

       Let $G$ be a nontrivial finite subgroup of the group ${\rm PSL}_2=\Aut ({\bf P}^1)=\Cl$.\;We consider the canonical homomorphism
  \begin{equation*}
  \nu\colon{\rm SL}_2\to {\rm PSL}_2
  \end{equation*}
  whose kernel is  the center $Z:=(-{\rm id})$.\;The group
  \begin{equation}\label{bing}
  \widetilde G:=\nu^{-1}(G)\subset {\rm SL}_2
  \end{equation}
  is either a binary rotation group of one of the regular polyhedra (dihedron, tetra\-hed\-ron, octahedron, or icosahedron), or a cyclic group of even order $\geqslant 4$.

   The subset $\bA2_{\,\ast}:=\bA2\setminus o$ is open in $\bA2$ and stable with respect to the actions on $\bA2$ of the groups $\widetilde G$ and $T:=\{(tx_1, tx_2)\mid t\in k^\times\}$.\;Let
 \begin{equation*}
 \pi\colon\bA2_{\,\ast}:=\bA2\setminus o
 \to \bf P^1
 \end{equation*}
 be the natural projection.\;The pair $(\pi, {\bf P}^1)$ is
 a geometric quotient for the action of the torus
 $T$ on $\bA2_{\,\ast}$.\;The morphism $\pi$ is $\widetilde G$-equivariant if we assume that the action of  $\widetilde G$ on  $\bf P^1$ is the restriction on
 $\widetilde G$ of the homomorphism $\nu$ (this action is not faithful, its kernel is $Z$).

  If a self-compression
 \begin{equation}\label{hp}
 \widetilde\varphi=({\widetilde\varphi}_1, {\widetilde\varphi}_2)\colon \bA2\dashrightarrow \bA2
 \end{equation}
  of  the group $\widetilde G$ is polynomial homogeneous of degree $d$,
  then the morphism
$\pi\circ \widetilde\varphi$ is constant on the  $T$-orbits in $\bA2_{\,\ast}$
and, therefore,
factors through $\pi$, i.e.,\;there is a morphism
 \begin{equation}\label{gd}
 \varphi\colon \bf P^1\to \bf P^1,
  \end{equation}
  such that $\varphi\circ\pi=\widetilde\varphi\circ\pi$.\;It is dominant
  (and therefore surjective) in view of the dominance of the morphism $\widetilde \varphi$.
 From the $\widetilde G$-equivariance of the morphisms $\pi$ and $\widetilde \varphi$ it follows the
 $\widetilde G$-equivariance\,---\,and therefore the $G$-equivariance\,---\,of the morphism $\varphi$.
 Consequently, $\varphi$ is the self-compression of the natural action of the group
 $G$ on $\mathbf P^1$.\;We say that the {\it self-compression $\varphi$ is a descent of the self-compression $\widetilde\varphi$}.

 \begin{theorem}\label{Ps} Let $G$ be a nontrivial finite subgroup of the Cremona group
 $\Cl={\rm PSL}_2=\Aut ({\bf P}^1)$.\;Associate with it the formal power series
 \begin{equation}\label{SG}
 S_G(t)=
 \sum_{n\geqslant 0}s_nt^n\in \bf Z[[t]]
\end{equation}
of the following form:
 \begin{enumerate}[\hskip 4.2mm \rm(a)]
 \item If $G$ is a rotation group of  tetrahedron, octahedron, or icosahedron, then
  \begin{equation}\label{SGPo}
 S_G(t)=
 t^{2a-1}(1+t^{4a-6})
 \sum_{n\geqslant 0}t^{2na}
 \sum_{n\geqslant 0}t^{(4a-4)n}+
t^{4a-5}
\sum_{n\geqslant 0} t^{(4a-4)n},
 \end{equation}
 where, respectively, $a=3, 4,$ or $6$.
  \item If $G$ is either a dihedral group of order $2\ell\geqslant 4$ or a cyclic group of order $\ell\geqslant 2$, then
 \begin{equation}\label{SDPo}
 S_G(t)=
 \sum_{n\geqslant 0}t^{2\ell(n+1)-1}.
 \end{equation}
  \end{enumerate}
Suppose that the coefficient  $s_d$ of the series {\rm \eqref{SG}}
is different from zero.\;Then
there exists a polynomial homogeneous self-compression {\rm \eqref{hp}}
of the binary group
 $\widetilde G$
 {\rm(}see {\rm \eqref{bing})}, whose degree is $d$, and
descent  {\rm \eqref{gd}}
 is a nontrivial self-compression of the group $G$.
 \end{theorem}

The proof of Theorem  \ref{Ps} will be given in  Subsection \ref{proof},
after proving several necessary auxiliary statements
in Subsection \ref{ac}.

\subsection{Application:\;self-compressibility of finite subgroups of \boldmath$\Cl$}
Theo\-rem \ref{Ps} im\-mediately implies statement (i)  of the following theorem.

\begin{theorem}\label{Cr1} \
 \begin{enumerate}[\hskip 2.2mm\rm(i)]
 \item
 Every finite subgroup of  $\Cl$ is nontrivially self-compressible.
 \item Every compression of a finite subgroup of  $\Cl$ is a compression ${\mathbf P}^1\to {\mathbf P}^1$ into a conjugate subgroup.
 \end{enumerate}
 \end{theorem}
  \begin{proof}[Proof of {\rm(ii)}]
Since every variety, to which $\mathbf P^1$ maps dominantly, is rational, (ii)
fol\-lows from the definition of compression and
the well-known fact that
two finite subgroups of $\Cl$ are isomorphic if and only if they are conjugate.
   \end{proof}

   \begin{remark} {\rm Another proof of statement (i) of Theorem \ref{Cr1} is given in \cite[Cor.\,1.3]{GA16}.
   This proof consists of presenting explicit formulas,
    in relation to which the reader
    is supposed to verify by direct computations that they define
   $G$-equivariant maps.\;In \cite{GA16}
    there are no comments about the origin of these formulas.\;For example (see \cite[Lemma 9.7]{GA16}), if  $\omega_5\in k $ is a primitive fifth root of 1 and $G$ is the lying in $\Cl$ rotation group of the icosahedron generated by the fractional-linear transformations
   \begin{equation*}
   \omega_5z \quad \mbox{and}\quad \frac{(\omega_5+\omega_5^{-1})z+1}{z-(\omega_5+\omega_5^{-1})},
   \end{equation*}
   then such a formula has the the appearance
   \begin{equation*}
 {\bf P}^1\to {\bf P}^1,\quad (x:y)\mapsto (x^{11}+66x^6y^5-11xy^{10}:-11x^{10}y-66x^5y^6+y^{11}).
\end{equation*}

Below  (see Remark (c) in Subsection \ref{comment}) we explain
how in principle the explicit formulas can be found that define any
self-compression specified in
Theorem\;\ref{Ps}.
}
   \end{remark}

   \subsection{Notations} To prove Theorem \ref{Ps} we need several notations.

We denote by ${\widetilde G}^\vee$
{\it  the set of characters of all simple  $k{\widetilde G}$-modules}.

The action of the group ${\widetilde G}$ on the affine plane $\bA2$
defines on the algebra
$$A:=k[\bA2]=k[x_1, x_2]$$
the structure of  a $k\widetilde G$-module.\;The latter is graded: each space
$A_n$ is its $k\widetilde G$-sub\-mo\-du\-le.

We denote by $\chi$ {\it the character of the submodule $A_1$}.\;If the group $\widetilde G$ is not cyclic, this submodule is simple.

For any simple  $k\widetilde G$-module $M$ with character $\gamma\in {\widetilde G}^\vee$, we denote by $A(\gamma)$  the {\it isotypic component of type $M$ in the $k\widetilde G$-module $A$}; it is its graded submodule.
In particular, $A(1)$ is the subalgebra of ${\widetilde G}$-invariants in $A$.

We will also need the following set of characters:
\begin{equation}\label{mchii}
[\chi]:=\{\gamma\in {\widetilde G}^\vee\mid \dim(\gamma)=1, \gamma\chi=\chi\}.
\end{equation}
It is not empty, because $1\in [\chi]$.

For any finite-dimensional  $k\widetilde G$-module $L$, we put
\begin{equation}\label{mu}
{\rm mult}_\chi (L):=\max\{d
\mid\mbox{there exists an embedding of $k\widetilde G$-modules
$A_1^{\oplus d}\hookrightarrow L$}\}.
\end{equation}

\subsection{Auxiliary statements}\label{ac} We now prove several auxiliary
statements that are used in the proof of Theorem \ref{Ps}.

\begin{lemma}\label{exist}
Let $H$ be a subgroup of ${\rm GL}_n$ and let $L$ be a finite-dimensional $k$-linear subspace in $k[x_1,\ldots, x_n]$.
\begin{enumerate}[\hskip 4.2mm\rm(a)]
\item The following conditions are equivalent:
\begin{enumerate}[\hskip 0mm$\rm(a_1)$]
\item[$(\rm a_1)$] $L$ is a submodule of the $kH$-module $k[x_1,
\ldots, x_n]$ that is isomorphic to the $kH$-module $k[x_1,
\ldots, x_n]_1.$
\item[$(\rm a_2)$] There exists a basis $\sigma_1,\ldots, \sigma_n$ of the linear space $L$ such that the mor\-phism
    $\sigma:=(\sigma_1,\ldots, \sigma_n)\colon
    \bAn\to \bAn$ is $H$-equiavriant.
\end{enumerate}
\item Suppose that the equivalent conditions $(\rm a_1)$ and $(\rm a_2)$ hold.
\begin{enumerate}[\hskip 0mm$\rm(b_1)$]
\item[$(\rm b_1)$] The morphism $\sigma $ from $(\rm a_2)$ is dominant if
and only if  $\sigma_1,\ldots, \sigma_n$ are algebraically independent over $k$.
\item[$(\rm b_2)$] If $n=2$ and $\sigma_1, \sigma_2\in k[x_1, x_2]_d$
for some $d$, then $\sigma_1,\sigma_2$ are algebraically independent over $k$.
\end{enumerate}
\end{enumerate}
\end{lemma}

\begin{proof} $(\rm a_1)\hskip -1mm\Rightarrow\hskip -1mm(\rm a_2)$: Let $k[x_1,\ldots, x_n]_1\to L$ be an isomorphism of  $kH$-modules and let $\sigma_i$ be the image of $x_i$ with respect to this isomorphism. Then the $H$-equivariance of $\sigma$ follows directly from the definitions and formulas \eqref{ast}, \eqref{string}.

$(\rm a_2)\hskip -1mm\Rightarrow\hskip -1mm(\rm a_1)$: It follows from \eqref{ast}, \eqref{string} that the restriction of
$\sigma^*$ to $k[x_1,\ldots, x_n]_1$ is an isomorphism of linear spaces $k[x_1,\ldots, x_n]_1\to L$.
From this and the $H$-stability of $k[x_1,\ldots, x_n]_1$ it follows the $H$-stability of $L$, so this restriction is an isomorphism of $kH$-modules.

$(\rm b_1)$:\;The dominance of $\sigma$ is equivalent to the triviality of the kernel of the homomorphism
$\sigma^*$ of the algebra $k[x_1,\ldots, x_n]$,  which, in view of  \eqref{string},
is equivalent to the algebraic independence of  $\sigma_1,\ldots, \sigma_n$ over $k$.

$(\rm b_2)$:\;Suppose that $\sigma_1$, $\sigma_2$ are algebraically dependent over $k$, i.e.,\; there exists a nonzero polynomial $F=F(t_1, t_2)\in k[t_1, t_2]$, where $t_1, t_2$ are variables, such that
\begin{equation}\label{F0}
F(\sigma_1, \sigma_2)=0.
\end{equation}
Since $\sigma_1$ and $\sigma_2$ are forms in $x_1, x_2$ of the same degree, we can (and shall) assume that $F$ is a form in $t_1, t_2$, say, of degree $s$:
\begin{equation}\label{form}
F(t_1, t_2)=\sum_{i=0}^s\alpha_it_1^{s-i}t_2^{i},\quad \alpha_0,\ldots, \alpha_s\in k.
\end{equation}

In view of $(\rm a_2)$, the polynomial $\sigma_2$
is nonzero, so that we can consider the rational function
$\sigma_1/\sigma_2\in k(x_1, x_2)$.
From \eqref{F0}, \eqref{form} we get for it the relation
\begin{equation}\label{form2}
0=\sum_{i=0}^s\alpha_i\Big(\frac{\sigma_1}{\sigma_2}\Big)^{\hskip -1mm {s-i}}.
\end{equation}

It follows from the linear independence of the polynomials $\sigma_1, \sigma_2$ over $k$ that the rational function  $\sigma_1/\sigma_2$ is not an element of $k$, and therefore
takes on $\bA2$ infinitely many different values. In view of \eqref{form2},
each of these values is the root of the nonzero polynomial
$\sum _{i=0}^s\alpha_it^{s-i}\in k[t]$, где $t=t_1/t_2$.\;This contradiction proves
the algebraic in\-de\-pen\-dence of $\sigma_1, \sigma_2$ over\;$k$.
\end{proof}

%%%%%%%%%%%%%%%%%%%

\begin{lemma}\label{nontriv} Let {\rm \eqref{hp}} be a polynomial homogeneous self-compression of the group $\widetilde G$, whose degree is $d$.\;Let a form $a\in A$ be the greatest common divisor of the forms ${\widetilde \varphi}_1$ and ${\widetilde \varphi}_2$ that define {\rm \eqref{hp}}. The following properties are equivalent:
\begin{enumerate}[\hskip 4.2mm\rm(a)]
\item
the descent {\rm\eqref{gd}} of the self-compression {\rm \eqref{hp}} is trivial;
\item $\deg(a)=d-1$;
\item there exists a character $\gamma\in [\chi]$ and an element $s\in A(\gamma)_{d-1}$ such that
   \begin{equation}\label{=a}
     {\widetilde\varphi}^*(A_1)=sA_1.
   \end{equation}
\end{enumerate}
\end{lemma}
\begin{proof}
(a)$\Leftrightarrow$(b):
If we consider $x_1$ and $x_2$ as homogeneous coordinates on $\mathbf
P^1$, then the self-compression \eqref{gd} is given by the formula
 \begin{equation}\label{frac}
 \varphi=\Big(\frac{{\widetilde \varphi}_1}{a}: \frac{{\widetilde\varphi}_2}{a}\Big)
 \end{equation}
  (see \cite[Chap.\;III, \S1, 4]{Sh13}).  Since every $k$-automorphism of the field of rational functions in one variable over  $k$ is a fractional linear transformation \cite[\S73]{Wa67}, it follows from \eqref{frac} and the inclusion $\widetilde \varphi_1, \widetilde\varphi_2\in A_d$ that the self-compression  $\varphi$ is trivial
  if and only if the forms ${{\widetilde \varphi}_1}/{a}$ и ${{\widetilde \varphi}_2}/{a}$ are linear, i.e., (b) is satisfied.

 (b)$\Leftrightarrow$(c):
Suppose that (b) holds. Then the equality
\begin{equation}\label{<>}
{\widetilde\varphi}^*(A_1)=\langle \widetilde \varphi_1, \widetilde \varphi_2\rangle
\end{equation}
and the definition of the form $a$ imply the equality \eqref{=a} for $s=a$.
In view of the
$\widetilde G$-invariance of the subspaces ${\widetilde\varphi}^*(A_1)$ and $A_1$, for every $g\in \widetilde G$, we obtain from  \eqref{=a} the following equalities:
\begin{equation}\label{dV*}
a A_1
=g\cdot (a A_1)=(g\cdot a)(g\cdot A_1)=(g\cdot a)A_1.
\end{equation}

We take some linear form $l\in A_1$, whose zero in $\mathbf P^1$
does not coincide with any of zeros of the form $a$. Since, in view of \eqref{dV*}, the form $(g\cdot a)l$ is divisible by $a$, and $\deg (g\cdot a)=\deg (a)$, this means that the divisors of the forms $a$ and $g\cdot a$ on $\mathbf P^1$ coincide, hence
$\langle a\rangle=\langle g\cdot a\rangle$. Therefore,
$a$ is a semi-invariant of the group $\widetilde G$.
Let $\gamma\in {\widetilde G}^\vee$ be the character of the one-dimensional $k\widetilde G$-module $\langle a\rangle$. Then $a\in A(\gamma)_{d-1}$, and
$\gamma\chi$ is the character of the
$k\widetilde G$-module $a A_1$. But since the $k\widetilde G$-modules $A_1$ and ${\widetilde\varphi}^*(A_1)$ are isomorphic, it follows from \eqref{=a} that the character of the
$k\widetilde G$-module $a A_1$ is $\chi$. Therefore, $\gamma\in [\chi]$. This proves  (b)$\Rightarrow$(c).

Conversely, if (c) is satisfied, then it follows from \eqref{=a} and \eqref{<>}
that
$s$ is the greatest common divisor of the forms  ${\widetilde\varphi}_1$ and
${\widetilde\varphi}_2$, and therefore $\langle s\rangle= \langle a\rangle$, and hence (b) is satisfied. This proves  (c)$\Rightarrow$(b).
\end{proof}

\begin{lemma}\label{grass} Let $H$ be a group,let $L$ be a nonzero $kH$-module of dimension $s<\infty$, and let $m$ be a positive integer.\;The Grassmanninan ${\rm Grass}(s, L^{\oplus m})$ contains a closed irreducible $(m-1)$-dimensional subset such that all the $s$-dimensional linear subspaces of the $kH$-module $L^{\oplus m}$
corresponding to its points are the submodules isomorphic to $L$.
 \end{lemma}
 \begin{proof} We assign to any nonzero vector
 $(\lambda_1,\ldots, \lambda_m)\in k^m$ the following embedding of the
 $kH$-modules:
\begin{equation*}\label{inj1}
\iota_{(\lambda_1,\ldots, \lambda_m)}\colon  L\hookrightarrow L^{\oplus m},\quad v\mapsto (\lambda_1v,\ldots, \lambda_mv).
\end{equation*}
The images of the embeddings $\iota_{(\lambda_1,\ldots, \lambda_m)}$  and $\iota_{(\mu_1,\ldots, \mu_m)}$ coincide if and only if the vectors $(\lambda_1,\ldots, \lambda_m)$ and $(\mu_1,\ldots, \mu_m)$ are proportional, i.e., the corresponding points $(\lambda_1:\ldots: \lambda_m)$ and $(\mu_1:\ldots: \mu_m)$ of the projective space $\mathbf P(k^m)$ coincide. Conse\-quent\-ly, the mapping
$\mathbf P(k^m)\to {\rm Gr}(s, L^{\oplus m}),$ which assigns to every point
$(\lambda_1:\ldots\break \ldots: \lambda_m)\in \mathbf P(k^m)$
the image of the embedding
$\iota_{(\lambda_1,\ldots, \lambda_m)}$,
is injective. It is not difficult to see that this mapping is a morphism. Therefore, its image is an irreducible closed subset of dimension $\dim (\mathbf P(k^m))=m-1$. It is this image that should be taken as the subset specified in the formulation of Lemma \ref{grass}.
 \end{proof}

\begin{lemma}\label{l4}
If, for a positive integer $d$, the inequality
\begin{equation}\label{condit}
{\rm mult}_\chi (A_d)
>\max\{\dim_k (A(\gamma)_{d-1})\mid \gamma\in [\chi]\},
\end{equation}
holds, then the group $\widetilde G$ admits a polynomial homogeneous self-compression {\rm \eqref{hp}} of degree $d$, whose descent  {\rm\eqref{gd}} is nontrivial.
\end{lemma}
\begin{proof}
Assume that the inequality \eqref{condit} holds.
For the sake of brevity, put
\begin{equation}\label{mmm}
m:={\rm mult}_\chi (A_d).
\end{equation}
It follows from \eqref{condit} that $m>0$.

According to \eqref{mu}, in the $k\widetilde G$-module $A_d$ there exists a submodule $M$ isomorphic to $A_1^{\oplus m}$. In view of $\dim (A_1)=2$ and Lemma\ref{grass}, in ${\rm Grass}(2, M)$ there exists an irreducible closed subset $X$ such that
\begin{equation}\label{d1}
  \dim (X)=m-1
\end{equation}
and all the $2$-dimensional linear subspaces of $M$ corresponding to its points are the submodules isomorphic to
 $A_1$. Since $M$ is a linear subspace in  $A_d$, the variety ${\rm Grass}(2, M)$, and hence $X$ as well, is a closed subset of ${\rm Grass}(2, A_d)$.

It follows from the definition of the set $[\chi]$ (see \eqref{mchii}) that for every character $\gamma\in [\chi]$ and every nonzero element
 $s\in A(\gamma)_{d-1}$, the linear subspace $sA_1$ is a submodule of the $k\widetilde G$-module $A_d$ isomorphic to $A_1$. This submodule does not change when $s$ is multilied by nonzero elements of $k$, therefore, assigning the submodule $sA_1$ to the element $s$ defines a mapping ${\mathbf  P}(A(\gamma)_{d-1})\to {\rm Grass}(2, A_d)$. It is not difficult to see that it is a morphism. Hence its image $Y(\gamma)$ is an irreducible closed subset in ${\rm Grass}(2, A_d)$, and
\begin{equation}\label{d2}
\dim (Y(\gamma)) \leqslant \dim ({\mathbf  P}(A(\gamma)_{d-1}))=\dim_k (A(\gamma)_{d-1})-1.
\end{equation}

In view of the finiteness of the set $[\chi]$, it follows from \eqref{condit}, \eqref{mmm}, \eqref{d1}, \eqref{d2} that
\begin{equation}\label{d3}
X\setminus \textstyle\bigcup_{\gamma\in [\chi]} Y(\gamma)
\end{equation}
is a nonempty subset of the Grassmannian ${\rm Grass}(2, A_d)$.

We consider a point of the set  \eqref{d3} and the two-dimensional linear subspace $L$ in $A_d$ corresponding to it. Then it follows from the above properties of the sets $X$ and $Y(\gamma)$ that
\begin{enumerate}[\hskip 2.2mm\rm (i)]
\item $L$ is a submodule of the
$k\widetilde G$-module $A_d$ isomorphic to $A_1$;
\item there are no $\gamma\in [\chi]$ and $s\in A(\gamma)_{d-1}$
such that $L=sA_1$.
\end{enumerate}

In view of (i) and Lemma \ref{exist}, there is a basis $\widetilde\varphi_1$,  $\widetilde\varphi_2$ of $L$ such that
\eqref{hp} is a polynomial homogeneous self-compression of the group
 $\widetilde G$ of degree $d$, for which
$\widetilde\varphi^*(A_1)=L$. It follows from (ii) and Lemma \ref{nontriv} that the descent {\rm\eqref{gd}} of this self-compression is nontrivial.
\end{proof}

\subsection{Proof of Theorem \ref{Ps}}\label{proof} The plan of the proof of Theorem \ref{Ps} is as follows. For each noncyclic finite subgroup $G$ of ${\rm PSL}_2=\Aut ({\bf P}^1)=\Cl$, we exlicitly describe for $\widetilde G$ the set $[\chi]$ and the Poincar\'e series
\begin{equation}\label{SPG}
P(\chi, t):=\sum_{n\geqslant 1}({\rm mult}_{\chi}(A_n))t^n,\quad
P(\gamma, t):=\sum_{n\geqslant 0} (\dim_k(A(\gamma)_n))t^n,\;\mbox{where $\gamma\in [\chi]$}.
\end{equation}
Comparing the coefficients of these series, we show that if a coefficient $s_d$ of the series \eqref{SG} is nonzero, then the inequality \eqref{condit} holds, from which, according to Lemma \ref{l4}, the statement of Theorem теоремы \ref{Ps} for $G$ follows. The case of a cyclic finite subgroup  $G$ is reduced to that of the corresponding dihedral one.

 \begin{proof}[Proof of Theorem {\rm \ref{Ps}}]
 We consider separately three possible types of the group $\widetilde G$.

\vskip 1mm

{\it {\rm (a)}
$\widetilde G$ is a primitive subgroup of the group ${\rm SL}_2$, i.e., a binary tetrahedral, octahed\-ral, or icasahedral group.}

    In view of \cite[3.2(a)]{Sp87} and the definition of the set $[\chi]$ (see \eqref{mchii}), in this case we have
    \begin{equation}\label{chipo}
    [\chi]=\{1\}.
    \end{equation}
   From \cite[4.2]{Sp87} we obtain:
\begin{equation}\label{tci1}
\begin{split}
P(\chi, t)&=\frac{t+t^{2a-1}+t^{4a-5}+t^{6a-7}}{(1-t^{2a})(1-t^{4a-4})},\\
P(1, t)&=\frac{1+t^{6a-6}}{(1-t^{2a})(1-t^{4a-4})},
\end{split}
\end{equation}
where $a=3, 4$, and $6$ respectively for binary tetrahedral, octahedral, and icosahedral group.
From \eqref{SPG}, \eqref{tci1} we then deduce the following:
\begin{equation}\label{decompbin1}
\begin{split}
P(\chi, t)-tP(1,t)
&=\sum_{n\geqslant 1}({\rm mult}_{\chi}(A_n)-\dim_k(A(1)_{n-1}))t^n\\
&=t^{2a-1}\frac{1+t^{2a-4}+t^{4a-6}-t^{4a-4}}{(1-t^{2a})(1-t^{4a-4})}\\
&=t^{2a-1}\frac{1+t^{4a-6}}{(1-t^{2a})(1-t^{4a-4})}+t^{4a-5}\frac{1}{1-t^{4a-4}}\\
&=t^{2a-1}(1+t^{4a-6})\sum_{n\geqslant 0}t^{2na}\sum_{n\geqslant 0}t^{(4a-4)n}+
t^{4a-5}\sum_{n\geqslant 0} t^{(4a-4)n}\\
&\overset{\eqref{SGPo}}{=\hskip -1mm=}S_G(t).
\end{split}
\end{equation}

From \eqref{decompbin1} and \eqref{SG} we obtain that
$s_d={\rm mult}_{\chi}(A_d)-\dim_k(A(1)_{d-1})$ for every $d>0$.
In view of \eqref{chipo}, this gives
\begin{equation}\label{sddd}
s_d= {\rm mult}_\chi (A_d)
-\max\{\dim_k (A(\gamma)_{d-1})\mid \gamma\in [\chi]\}\;\;\mbox{for every $d$}.
\end{equation}
As \eqref{SGPo} shows,  if
$s_d\neq 0$, then $s_d>0$. In view of \eqref{sddd}
and Lemma \ref{l4}, this implies the statement of Theorem \ref{Ps} in case (a).

 {\it {\rm (b)}
$\widetilde G$ is an irreducible imprimitive subgroup of the group ${\rm SL}_2$, i.e., a binary dihedral group of order $4\ell\geqslant 8$.}

In this case, the McKay correspondence \cite[Sect.\;2]{Sp87} juxtaposes to the group
$\widetilde G$ the extended Dynkin diagram of type
${\sf D}_{\ell+2}^{(1)}$ with $\ell+3$ vertices. According to \cite[2.3(a)]{Sp87}, the vertex juxtaposed to the character $\chi$
is a branch point of this diagram.
In ${\widetilde G}^\vee$ there are exactly four one-dimensional characters $1$, $\theta$, $\theta'$, $\theta''$  (see\;\cite[4.3]{Sp87}); the vertices of the diagram juxtaposed to them is the set of all its endpoints.\;In view of \cite[2.2]{Sp87} and \eqref{mchii},
an endpoint corresponds to a character from
$[\chi]$
if and only if it is connected by an edge to the vertex juxtaposed to the character $\chi$. Therefore, apart from $1$, there is at least one more character in $[\chi]$ (we denote it by $\theta$), and two possibilities occur:

---\;If $\ell\geqslant 3$, then the vertex juxtaposed to the character $\chi$ is connected by edges to only two endpoints of the diagram, which are juxtaposed to the one-dimensional characters $1$ and $\theta$:
\begin{equation}\label{dl11}
\hskip -48mm\begin{matrix}
\xymatrix@=4.5mm@M= -.4mm@R=3mm{ %%{}^{1}
\circ\ar@{-}[dr]&&&&\circ
\\
&\circ\ar@{-}[r]^{\hskip -8.22mm \chi} & {\ \cdots \
}\ar@{-}[r]^{\hskip 6mm \alpha} & \circ\ar@{-}[ur] \ar@{-}[dr]
\\
%%{}^{\theta}
\circ\ar@{-}[ur]&&&&\circ }
\end{matrix}\quad,\quad \chi\neq \alpha.
\hskip -21.3mm\raisebox{1.8\height}{\mbox{\tiny $\theta'$}}
\hskip -2.5mm\raisebox{-2.6\height}{\mbox{\tiny $\theta''$}}
\hskip -35.9mm\raisebox{-2.9\height}{\mbox{\tiny $\theta$}}
\hskip -1.5mm\raisebox{1.8\height}{\mbox{\tiny $1$}}
\end{equation}
Thus, we obtain
\begin{equation}\label{chidie3}
[\chi]=\{1, \theta\}\quad\mbox{for $\ell\geqslant 3$}.
\end{equation}

---\;If $\ell=2$, then the vertex juxtaposed to the character $\chi$ is connected by edges to all four endpoints, which are juxtaposed to the one-dimensional characters $1$, $\theta$, $\theta'$ и $\theta''$:
\begin{equation*}
\hskip -2mm\begin{matrix}
\xymatrix@C=3.1mm
@M= -.1mm
@R=3.1mm{
%%{}^{1}
\circ\ar@{-}[dr]&&\circ
\\
&\circ
\ar@{-}[ur] \ar@{-}[dr]&
\\
%%{}^{\gamma}
\circ\ar@{-}[ur]&&\circ }
\end{matrix}
\quad {\hskip -8mm\mbox{\tiny $\chi$}}
\hskip -11.7mm\raisebox{2.1\height}{\mbox{\tiny $1\hskip 13mm \theta'$}}
\hskip -17.0mm\raisebox{-2.6\height}{\mbox{\tiny $\theta\hskip 13mm\theta''$}}\;.
\end{equation*}
Therefore, we obtain
\begin{equation}\label{chidie4}
[\chi]=\{1, \theta, \theta', \theta''\}\quad\mbox{for $\ell=2$}.
\end{equation}

It follows from \cite[(9)]{Sp87} that, for every $\ell\geqslant 2$, $\theta$ is the character denoted in \cite[p.\;103]{Sp87} by $c_1$. From this and
 \cite[4.4]{Sp87} we obtain:
\begin{equation}\label{di1}
\left.\begin{split}
P(\chi, t)&:=\frac{t+t^3+t^{2\ell-1}+t^{2\ell+1}}{(1-t^4)(1-t^{2\ell})},\\
P(1, t)&:=\frac{1+t^{2\ell+2}}{(1-t^4)(1-t^{2\ell})},\\
P(\theta, t)&=\frac{t^2+t^{2\ell}}{(1-t^4)(1-t^{2\ell})}.
\end{split}\right\}\quad\mbox{for every $\ell\geqslant 2$}.
\end{equation}
In addition, according to \cite[4.4]{Sp87},
\begin{equation}\label{3theta}
P(\theta, t)=P(\theta', t)=P(\theta'',t)\quad\mbox{for $\ell=2$}.
\end{equation}

From \eqref{SPG}, \eqref{di1} we obtain:
\begin{equation}\label{decompbin2}
\begin{split}
P(\chi, t)-tP(1,t)&-tP(\theta, t)\\
&=\sum_{n\geqslant 1}\big({\rm mult}_{\chi}(A_n)-\dim_k(A(1)_{n-1})-\dim_k(A(\theta)_{n-1})\big)t^n\\
&=t^{2\ell-1}\frac{1-t^4}{(1-t^4)(1-t^{2\ell})}
=\sum_{n\geqslant 0}t^{2\ell (n+1)-1}\\
&\overset{\eqref{SDPo}}{=\hskip -1mm=}S_G(t).
\end{split}
\end{equation}

It follows from \eqref{decompbin2} and \eqref{SG} that
$$s_d={\rm mult}_{\chi}(A_d)-\dim_k(A(1)_{d-1})-\dim_k(A(\theta)_{d-1})\quad
\mbox{for every $d>0$.}
$$
In view of \eqref{chidie3}, \eqref{chidie4},and \eqref{3theta}, this gives
\begin{equation}\label{sdd}
s_d\leqslant {\rm mult}_\chi (A_d)
-\max\{\dim_k (A(\gamma)_{d-1})\mid \gamma\in [\chi]\}\;\;\mbox{for every $d$}.
\end{equation}
As \eqref{SDPo} shows,
if $s_d\neq 0$, then $s_d> 0$. In view of \eqref{sdd} and Lemma  \ref{l4}, this implies the statement of Theorem
\ref{Ps} in case (b).

\vskip 1mm

 {\it {\rm (c)}
$\widetilde G$ is a cyclic subgroup of order
$2\ell\geqslant 4$ in the group ${\rm SL}_2$.}

Since  $\widetilde G$ is a subgroup of a binary dihedral group of order $4\ell$, every polynomial homogeneous self-compression of the latter group is a self-compression of the group
$\widetilde G$. Therefore, the statement of Theorem  \ref{Ps} for $\widetilde G$  follows from its statement, already proved, for binary dihedral groups.
 \end{proof}

 \subsection{Remarks about Theorems \ref{Ps} and \ref{Cr1}}\label{comment} Concluding the discussion of self-compressibility of finite subgroups of the Cremona group of rank 1, we will make a few remarks about Theorems  \ref{Ps} and \ref{Cr1}.

\vskip 1mm

(a) For every nontrivial finite subgroup $G$ of the Cremona group $\Cl$, Theorem \ref{Ps} yields an infinite set of natural numbers
$d$, for which there exists a polynomial homogeneous self-compression {\rm \eqref{hp}} of the corresponding binary group
 $\widetilde G$, whose degree is $d$ and descent
 {\rm \eqref{gd}} is a nontrivial self-compression of the group\;$G$. From formulas \eqref{SGPo}, \eqref{SDPo} we obtain the minimal of these $d$: it is equal to
 5, 7, and 11 respectively for the rotation group of a regular tetrahedron, octahedron, and icosahedron, and to
 $2\ell-1$ for the dihedral group of order $2\ell\geqslant 4$ and cyclic group of order $\ell\geqslant 2$.

\vskip 1mm

(b) Let $K$ be a field of algebraic functions in one variable over $k$.\;By Theorem \ref{Cr1}(i), if the genus of the field $K$ is equal to $0$, then every finite subgroup of ${\rm Aut}_k(K)$ is nontrivially self-compressible.\;For the fields $K$ of genus $\geqslant 2$, this, in general, is not so, see \cite[Ex.\;6]{Re04}, \cite{GA16}.

\vskip 1mm

(c)  Let us briefly explain how one can find explicitly the forms
 $\widetilde \varphi_1, \widetilde\varphi_2$, which define the homogeneous polynomial self-compressions \eqref{hp} of the group $\widetilde G$, whose degrees are specified in Theorem \ref{Ps}.

Let $d$ be such a degree.\;In view of what was said in part (c) of the proof of Theorem {\rm \ref{Ps}}, we can (and shall) assume that the group $\widetilde G$ is not cyclic, and hence the character $\chi$ is irreducible.\;Consider the linear space  ${\mathscr L}(A_1, A_d)$ of all linear maps $A_1\to A_d$. The group $\widetilde G$  acts linearly on it by the rule
\begin{equation}\label{line1}
(g\ell)(a):=g(\ell(g^{-1}(a))),\quad g\in \widetilde G,\; \ell\in{\mathscr L}(A_1, A_d),\; a\in A_1,
\end{equation}
and the $\widetilde G$-equivariant maps are precisely the fixed points of this action.\;They form in ${\mathscr L}(A_1, A_d)$ a linear subspace
${\mathscr L}(A_1, A_d)^G$, see \cite[3.12]{PV94}.\;The latter is the image of the Reynolds operator $|\widetilde G |^{-1}\sum_{g\in\widetilde G}g$ for the action \eqref{line1}, see \cite[3.4]{PV94},  and therefore, is described effectively.
If $\ell_1,\ldots, \ell_m$\,is a basis of
${\mathscr L}(A_1, A_d)^G$, then
$\langle \bigcup_{i=1}^m\ell_i(A_1)\rangle=A(\chi)_d$. Similarly, effectively are described the isotypic components  $A(\gamma)_{d-1}$ for all $\gamma\in[\chi]$.

According to the proof of Theorem \ref{Ps}, the set of all $(\alpha_1,\ldots,\alpha_m)\in k^m$ such that $L:=(\sum_{i=1}^m\alpha_i\ell_i)(A_1)$ does  not lie in $\langle A(\gamma)_{d-1}A_1\rangle$ for all $\gamma\in [\chi]$, is
nonempty. Effective finding of such a $(\alpha_1,\ldots,\alpha_m)$
reduces to finding for some explicitly de\-scribed nonzero polynomial in
 $m$ variables with coefficients in $k$
 any values of these variables that do not make this polynomial zero.

The linear space $L$ is a
$k\widetilde G$-submodule of  $A_d$ isomorphic to $A_1$.\;As a couple of forms $\widetilde \varphi_1, \widetilde\varphi_2$
one can now take a basis of this subspace such
that the matrices of the elements of the group $\widetilde G$
in this basis are the same as in the basis $x_1, x_2$ of the space  $A_1$ (it suffices to ensure this only for the system of generators of the group $\widetilde G$, containing two elements for dihedral group and three for the others, see \cite{Sp87}).
Effective finding of such a basis
reduces to finding a solution of a system of linear equations for the coefficients of the transition matrix, which satisfies an inequality equivalent to the nondegeneracy of this matrix.

(d) Theorem  \ref{Cr1} naturally leads to the question
 of whether its statements (i) and (ii) will remain true
 if  $\Cl$ isreplaced by
  $\Cp$ in them.\;As is shown below (see Theorem \ref{rdim1av}),
 concerning statement (ii)
 the answer is negative.
  Concerning the statement (i), at the time of this writing (September 2018) the author does not know the answer, and the following question seems to him to be of a principal importance:

\vskip 2.5mm

\noindent{\bf Qustion {\rm (\cite[Quest.\,1]{Po16})}.} {\rm
Is there an incompressible rational action of a finite group on
$\bA2$?
}

\vskip 2.5mm

In the case of a positive answer to this question, the problem of finding all incompressible actions in the list found in \cite{DI09}
naturally arises.

\vskip 2mm

\subsection{Self-compressing linear actions} The remaining results of Section \ref{2} are divided into two groups: one relates to the general case,   the other to the case of
$\Cp$. The following theorem applies to the general case.

\begin{theorem}\label{coli} Let $G$ be a finite subgroup of ${\rm GL}_n$, $n\geqslant 1$.
\begin{enumerate}[\hskip 4.2mm\rm(a)]
\item If $k[x_1,\ldots, x_n]_d^G\neq 0$, then $G$ admits
a polynomial homogeneous self-com\-pres\-sion
$\bAn\to \bAn$ of degree $d+1$.\;For $d\neq 0$, its is nontrivial.
\item If $|G|$ divides $d$, then $k[x_1,\ldots, x_n]_d^G\neq 0$.
\end{enumerate}
\end{theorem}

\begin{proof}
(a) We take a nonzero polynomial $f\in
k[x_1,\ldots, x_n]_d^G$
and consider the mor\-phism
\begin{equation}\label{mor}
\varphi\colon \bAn\to \bAn,\quad a\mapsto f(a)a,
\end{equation}
In view of the   $G$-invariance of $f$ and the linearity of the action of $G$ on $\bAn$, for any
 $g\in G$ and $a\in\bAn$
we have
\begin{equation*}
\varphi(g\cdot a)
\overset{\eqref{mor}}{=\hskip -1mm=}
f(g\cdot a)(g\cdot a)=f(a)(g\cdot a)=
g\cdot \big(f(a)a\big)
\overset{\eqref{mor}}{=\hskip -1mm=}
g\cdot \big(\varphi(a)\big),
\end{equation*}
i.e., $\varphi$ is a $G$-equivariant morphism.
From \eqref{mor} and $f\in k[x_1,\ldots, x_n]_d$ we obtain
\begin{equation}\label{line}
\varphi(ta)=t^{d+1}f(a)a=t^{d+1}\varphi(a)\;\; \mbox{for any $a\in \bAn$, $t\in k$,}
\end{equation}
so $\varphi$ is a polynomial homogeneous map of degree $d+1$.

From \eqref{line} it follows that if a line $L$ in $\bAn$ contains $0$ and a point
$a\in U:=\{c\in \bAn\mid f(c)\neq 0\}$ different from $0$,
 then
\begin{enumerate}[\hskip 2.2mm\rm(i)]
\item $\varphi(L)=L$;
\item the degree of the morphism
$\varphi|_{L}\colon L\to L$ is equal to $d+1$.
\end{enumerate}
In view of (i), the image of  $\varphi$ contains a set $U$ open in $\bAn$, therefore, $\varphi$ is dominant, hence is a self-compression of the group $G$.

Suppose that $\varphi$ is a birational isomorphism.\;Then
the restriction of  $\varphi$ to some nonempty open subset $U'$ of $\bAn$ is injective.
Since $\bAn$ is irreducible, $U\cap U'\neq \varnothing$. Let $a\in U\cap U'$.
Then, in the previous notation, the degree of the morphism $\varphi|_{L}$  is equal to $1$ in view of
its injectivity
on the subset $L\cap U'$, which is open in $L$.
From\;(ii)
we then obtain $d=0$, which completes the proof of\;(a).

(b) The kernel of the natural action of $G$ on $k[x_1,\ldots, x_n]_1$ is trivial.\;Therefore there is a nonzero linear form $\ell\in  k[x_1,\ldots, x_n]_1$ such that its $G$-orbit contains exactly $|G|$ elements.\;Therefore, $\big(\prod_{g\in G}g\cdot \ell\big)^{\hskip -.5mm s}$ is a nonzero $G$-invariant form of degree $s|G|$ for any integer
$s\geqslant 0$, which proves\;(b).
\end{proof}

\begin{corollary}\label{cccco}%%{cornonumber}
Every finite subgroup of $\Cn$, which is conjugate to a subgroup of the group ${\rm GL}_n$, is nontrivially self-compressible.
\end{corollary}%%{cornonumber}

\subsection{Compressing actions with a fixed point}\label{fabe}
The following result is an application of Theorem \ref{coli}.

\begin{theorem}\label{fipo} Every {\rm(}faithful\,{\rm)} rational action $\varrho$ of a finite group $G$ on an $n$-dimen\-sional irreducible variety, which has a fixed point, is obtained by a nontrivial base change from a {\rm(}faithful\,{\rm)} linear action of the group $G$ on an $n$-dimensional linear space.
\end{theorem}

\begin{proof} Let $Y$ be an irreducible smooth complete variety and let $G\hookrightarrow {\rm Aut}(Y)$ be a regularization of the action $\varrho$, such that
$Y^G\neq \varnothing$. Let $y\in Y^G$.\;We consider a nonempty open affine subset $U$ of $Y$, containing
$y$. Since $\bigcap_{g\in G}g\cdot U$ is a $G$-stable open affine subset, which contains $y$, replacing $U$ by it, we can (and shall) assume that $U$
is $G$-stable. Since $U$ is dense in $Y$, the action of $G$ on $U$ is faithful.

Let ${\rm T}_{y, U}$ be the tangent space to $U$ at the point $y$. The tangent action
\begin{equation*}
\tau\colon G\to {\rm GL}({\rm T}_{y, U})\subset {\rm Bir}({\rm T}_{y, U})
\end{equation*}
of the group $G$ on the space ${\rm T}_{y, U}$ is faithful \cite[Lem.\;4]{Po14_1}.\;According to \cite[Lem. 10.3]{LЗК06}, there is a
$G$-equivariant dominant morphism
$\alpha\colon U\to  {\rm T}_{y, U}$.\;In view of Theorem \ref{coli},
the linearity of the action
$\tau$ implies the existence of a nontrivial self-compression
$\beta\colon {\rm T}_{y, U}\to {\rm T}_{y, U}$
of the group $G$
(so that\;$\deg (\beta)>1$).\;Then $\beta\circ\alpha\colon Y\!\dashrightarrow\! {\rm T}_{y, U}$
is a nontrivial (because $\deg (\beta\circ\alpha)=\deg(\beta)\deg(\alpha)\!>\!1$)
self-compression of the action\;$\varrho$.
\end{proof}

\begin{corollary}%%{cornonumber}
Every incompressible rational action of a finite group on an irre\-ducible variety has no fixed points.
\end{corollary}%%{cornonumber}

\begin{remark} {\rm In \cite{DD16} all rational actions of finite groups on $\bA2$, having a fixed point, are found.\;There are quite a few of them.\;By Theorem \ref{fipo} they all are obtained by nontrivial base changes from the linear actions on
 $\bA2$ (the classification of which has long been known,
see, e.g., \cite{NPT08}).}
\end{remark}

Recall that every finite Abelian group $G$ decomposes into a direct sum of cyclic
subgroups of orders $m_1,\ldots, m_r$, where $m_i$ divides $m_{i+1}$ for $i=1,\ldots, r-1$,
and $m_1>1$ for $|G|>1$.\;The sequence  $m_1,\ldots, m_r$ is uniquely determined by $G$ and  called
{\it the sequence of invariant factors of the group} $G$.\;The integer
$r$ is called its {\it rank}; the latter is equal to the minimal number of generators of the group
$G$.

For every integer $n\geqslant r$, we distinguish in
${\rm GL}_n\subset \Cn$ the following subgroup isomorphic to $G$:
\begin{equation}\label{D}
 T_n(m_1,\ldots, m_r)\!:=\!\{%%{\rm diag}
 (t_1x_1,\ldots, t_rx_r, x_{r+1},\ldots, x_n)%%\in {\rm GL}_n
 \mid
 t_i\!\in\! k,\, t_i^{m_i}=1, 1\leqslant i\leqslant r\}.
 \end{equation}

 \begin{theorem}\label{afipo}
  Let $G$ be a finite Abelian group with the sequence of invariant factors
 $m_1,\ldots, m_r$.\;If a
 {\rm(}faithful\,{\rm)} rational action  $\varrho$ of the group
 $G$ on an $n$-dimen\-sional irreducible variety has a fixed point, then $n\geqslant r$ and $\varrho$ is obtained by a nontrivial base change from a linear action $\lambda\colon G\!\hookrightarrow\! \GL_n(k)$ on $\bAn$, such that
 $\lambda(G)\!=\!
T_n(m_1,\ldots, m_r)$.
 \end{theorem}

 \begin{proof} We use the notation from the proof of Theorem \ref{fipo}.\;Fixing an isomorphism of the space ${\rm T}_{y, U}$ with $\bAn$,
  we identify the group ${\rm Bir}({\rm T}_{y, U})$ with $\Cn$.\;Since the groups $G$ and
  $\tau (G)$ are isomorphic, they have the same invariant factors. According to
 \cite[Thm.\;1]{Po13_2}, every finite Abelian subgroup of ${\rm Aff}_n$, whose invariant factors are
   $m_1,\ldots, m_r$, is transformed to the group
 $T_n(m_1,\ldots, m_r)$ by means of conjugation in the group $\Cn$.\;Hence this is true for the subgroup $\tau (G)$. From here,
  arguing as in the proof of Theorem \ref{fipo}, we get the statement to be proved.
 \end{proof}

 \subsection{Compressing actions of cyclic groups} According to \cite[Thm.\;A]{Bl06}, the sets of conjugacy classes of cyclic subgroups of $\Cp$ of some fixed orders
$d<\infty$ are infininite and even there are parameter-dependent families of such classes (this is the case if $d$ is even, $d/2$ is odd).\;The following theorem
(a special case of which is proved in \cite[Ex.\;5]{Re04})
implies that all these subgroups are obtained by the base changes from a single such subgroup.

 \begin{theorem}\label{rdim1av}
 Let $n, m, d$ be any positive integers, and $n\geqslant m$.\;Every {\rm(}faithful\,{\rm)} rational action $\varrho$
of a finite cyclic group $G$ of order $d$
on an $n$-dimensional irreducible variety $X$
is obtained by a nontrivial base change from a linear action
$\lambda\colon G\hookrightarrow \GL_m(k)$ on $\bAm$
such that
$\lambda(G)=T_m(d)$.
 \end{theorem}

 \begin{proof}\label{cyav}
Let $G^{\vee}$ be the group of all homomorphisms  $G\to k^\times$, and put
$K:=k(X)$. If $\chi\in G^{\vee}$, then
 \begin{equation}\label{eis}
 K_\chi:=\{f\in K\mid g\cdot f=\chi(g)f \;\mbox{for all $g\in G$}\}
 \end{equation}
 is a linear subspace of the linear space  $K$ over $k$.\;Since $\dim(\langle G\cdot f\rangle)<\infty$ and $G$ is Abelian, we have  $K=\bigoplus_{\chi\in G^{\vee}} K_{\chi}$, and \eqref{eis} implies that
 $\{\chi\in G^{\vee}\mid K_\chi\neq 0\}$ is a subgroup of
$G^{\vee}$. Since $G^{\vee}$ is a cyclic group of order $d$ and the action of $G$ on $K$ is faithful, this subgroup coincides with\;$G^\vee$.

Let $\chi_0^{\ }$ be a generator of  $G^{\vee}$ and let
 $f\in K_{\chi_0^{\ }}$, $f\neq 0$. We can (and shall) assume that  $f\notin k$ (and hence $f$ is transcendental over $k$):  if $d>1$, this holds automatically; if $d=1$, this follows from
 $K\neq k$.
 %% ввиду ${\rm trdeg}_kK=n>0$.

  Since $K/K^G$ is a finite field extension,  ${\rm trdeg}_kK^G={\rm trdeg}_kK=n$. Let $h_1,\ldots, h_n\in K^G$ be a transcendence basis of the field $K^G$ over $k$.
  Then ${\rm trdeg}_kk(f, h_1,\ldots, h_n)=n$. Since $f$ is transcendental over $k$, this and \cite[Chap.\;X, \S1, Thm.\;1]{La65} imply that after possibly renumbering elements $h_1,\ldots, h_n$ the subfield
 $L:= k(f, h_1,\ldots, h_{n-1})$ of $K$
  is a purely transcendental extension of the filed  $k$ of degree $n$. The construction implies the $G$-invariance of $L$ and the faithfulness of the action of
$G$ on $L$. It follows from the definition of $\chi_0$ that the linear action  $\mu\colon G\hookrightarrow \GL_n(k)$ on $\bAn$, given by the formula $g\cdot (a_1,\ldots, a_n):=(\chi_0^{\ }(g^{-1})a_1, a_2,\ldots, a_n)$, is faithful, $\mu(G)=T_n(d)$, and the map  $\varphi\colon X \dashrightarrow\bAn$ is a contraction of the action  $\varrho$ into the action $\mu$. Since, in turn, $\mu$ is compressed into a linear action of
the group $G$ on $\bAm$ by means of the projection$\An\to \Am$, $(a_1,\ldots, a_n)\mapsto (a_1,\ldots, a_m)$, the assertion being proved now follows from  Corollary \ref{cccco}.
 \end{proof}

 \subsection{Auxiliary statement:\;embeddings of \boldmath$G$-modules into coordinate al\-geb\-ras} In what follows we shall need the following general statement:

\begin{lemma}\label{lemma} If a finite group $G$ acts regularly  {\rm(}and faithfully\,{\rm)} on an irreducible affine variety $X$, then every finite-dimensional $\,kG$-module $M$ is isomorphic to a submodule of the $kG$-module $k[X]$.
\end{lemma}

\begin{proof} We can (and shall) assume that $\dim (X)>0$.\;Since ${\rm tr\,deg}_k (k(X)^G)=\dim (X)\break -\dim (G)=\dim (X)$ (see\;\cite[Sect.\;2.3,\;Cor.]{PV94}), and $k(X)^G$ is the field of fractions of the algebra $k[X]^G$ (see\;\cite[Lemma
3.2]{PV94}), the latter is an infinite-dimensional linear space over $k$:
\begin{equation}\label{inf}
\dim_k (k[X]^G)=\infty.
\end{equation}

From ${\rm char}\,k=0$ and the finiteness of the group $G$ it follows that the $kG$-modules $M$ and $k[X]$ are completely reducible.\;Therefore, to prove the lemma, it suffices to establish that, for every nonzero simple $kG$-module $S$, the multiplicity of its occurrence in the $S$-isotypic component of the $kG$-module $k[X]$ is infinite, which is equivalent to the infinite-dimensionality of this isotypic component as a linear space over $k$.\;In turn, for
this is sufficient to establish that this $S$-isotypic component is nonzero.\;Indeed,  the multiplication of functions defines on it a structure of a $k[X]^G$-module.\;Therefore, if this $S$-isotypic component contains a nonzero
function, its infinite-dimensionality follows from \eqref{inf} and the absence of zero-divisors in $k[X]$. Having in view this reduction, we shall now prove that the
$S$-isotypic component of the $kG$-module $k[X]$ is indeed nonzero.

The set of fixed points of every element of $G$ is closed in $X$.\;Since
$G$ is fnite, $X$ is irreducible, and the action of $G$ on $X$ is faithful, this implies that there exists a point $x$ of $X$, whose  $G$-stabilizer $G_x$ is trivial. Its
$G$-orbit
$G\cdot x$ is a $G$-stable and (in view of the finiteness) closed subset of $X$. Its closedness implies that the homomorphism of $kG$-modules
\begin{equation*}\label{rest}
k[X]\to k[G\cdot x],\quad f\mapsto f|_{G\cdot x},
\end{equation*}
is surjective. Therefore, to prove the nontriviality of the $S$-isotypic component of the $kG$-module $k[X]$, it suffices to prove the nontriviality of the
$S$-isotypic component of the $kG$-module  $k[G\cdot x]$.\;But it follows from the Frobenius duality that the multiplicity of the occurence of the
 $kG$-module $S$  in the $kG$-module $k[G\cdot x]$ is equal to the dimension of the space of
$G_x$-fixed points in the dual $kG$-module $S^*$
(see\;\cite[Thm.\;3.12]{PV94}). Since the group $G_x$ is trivial, this shows that  the specified multiplicity is equal to
$\dim (S)>0$.\;Therefore, the $S$-isotypic component of the
$kG$-module  $k[G\cdot x]$ is indeed nonzero. This completes the proof
of Lemma\;\ref{lemma}.
\end{proof}

\subsection{\boldmath${\rm rdim}_ k(G)$ and the existence of compressions} For any finite group $G$ and any field $\ell$ we put
\begin{equation}\label{rdim}
{\rm rdim}_ \ell(G):=\min \{m\!\in\! \mathbb Z, m\!>\!0\mid
\mbox{there is a group embedding $G\!\hookrightarrow\! \GL_m(\ell)$}\!\}.
\end{equation}
In other words, ${\rm rdim}_\ell(G)$ is the minimum of dimensions of faithful linear representa\-tions of the group $G$ over the field $\ell$.\;Thus $G$ has a faithful $n$-dimensional linear representation over $\ell$ if and only if $n\geqslant {\rm rdim}_ \ell(G)$.\;Note that if the group $G$ is Abelian,
then ${\rm rdim}_k(G)$ is equal to its rank.

 \begin{theorem}\label{nrdi} Let $\varrho$ be a {\rm(}faithful\,{\rm)} rational action of a finite group $G$ on an $n$-dimensional irreducible variety.
\begin{enumerate}[\hskip 4.2mm \rm(i)]
\item  If $\varrho$ has a fixed point, then $n\geqslant
{\rm rdim}_k(G)$.
\item  If $n>
{\rm rdim}_k(G)$, then $\varrho$ is compressible into a {\rm(}faithful{\rm)} rational action of a smaller dimension.
\item If there is an $n$-dimensional faithful linear representation  over $k$
    \begin{equation}\label{lrep}
    \lambda\colon G\to {\rm GL}_n
        \subset \Autn,
     \end{equation}
      then either $\varrho$ is compressible into a {\rm(}faithful\,{\rm)} rational action of a smaller dimension or $\varrho$ is obtained by a nontrivial base change from a linear action $\lambda$ of the group $G$ on $\bAn$.
\end{enumerate}
  \end{theorem}
  \begin{proof} Consider a regular  (faithful) action of the group
  $G$ on an $n$-dimensional smooth variety $X$, which is a regularization of the action $\varrho$.

  (i) If $\varrho$ has a fixed point, we choose $X$ so that $X^G\neq\varnothing$. Let $x\in X^G$. Since the tangent action
\begin{equation}\label{tan}
G\to {\rm GL} ({\rm T}_{x, X})
\end{equation}
of the group $G$ on ${\rm T}_{x, X}$ is faithful \cite[Lem.\;4]{Po14_1}, the homomorphism
\eqref{tan} is injective. From this and  \eqref{rdim} it follows that $n=\dim (X)=\dim ({\rm
T}_{x, X})\geqslant {\rm rdim}_k(G)$.
\vskip 1mm

(ii) In view of the inequality ${\rm ed}_k(G)\leqslant {\rm rdim}_k(G)$, which follows from
\eqref{rdim} and the definition of ${\rm ed}_k(G)$   (see
\cite[Thm. 3.1(b)]{BR97}), the statement
(ii) follows from the inequality
${\rm ed}_k(X)\leqslant {\rm ed}_k(G)$ proved in
 \cite[Thm. 3.1(c)]{BR97}.

Other proof of the statement (ii), not using
\cite[Thm. 3.1(b, c)]{BR97}, is obtained in the course of the proof of (iii) below, see
Remark \ref{op}.

\vskip 1mm

(iii) As in the proof of Theorem \ref{fipo}, replacing $X$  by an appropriate
invariant open subset, in the sequel we can (and shall) assume that
$X$ is affine.

Since the representation $\lambda$ is faithful, the dual representation
$\lambda^*\colon G\to {\rm GL}_r(k)$ is faithful as well.\;From
%%\eqref{rdim} and
Lemma \ref{lemma} it follows that there a linear subspace $L$ in $k[X]$ with the following properties:
\begin{enumerate}[\hskip 4.2mm\rm (a)]
\item\label{a} $L$ is $G$-stable;
\item\label{c} $\dim (L)= n$;
\item\label{b} the action of $G$ on $L$ is the representation $\lambda^*$.
\end{enumerate}

Consider in $k[X]$ the $k$-subalgebra $A$ generated by the subspace $L$.  Since $\dim (L)<\infty$, it is finitely generated and therefore isomorphic to the algebra of regular functions of an affine variety $Y$.  It follows from \eqref{c} that
\begin{equation}\label{dimY}
\dim (Y)\leqslant n.
\end{equation}
The identity embedding $A\hookrightarrow  k[X]$ determines a dominant morphism
$\varphi\colon X\to Y.$ From \eqref{a} the  $G$-invarince of $A$ follows.
The action of
$G$ on
$A$ determines a regular action $\vartheta$ of the group $G$ on the variety $Y$.\;The morphism $\varphi $ is $G$-equivariant with respect to $\vartheta$. In view of \eqref{b} and the faithfulness of the representation $\varrho^*$, the action
 $\vartheta$ is faithful. Therefore, $\varphi$ is a compression of $\varrho$ into $\vartheta$.

Suppose that this compression
does not reduce the dimension of the action $\varrho$,
i.e.,
%%$\dim (X)=\dim (Y)$. Then
\begin{equation}\label{Yr}
\dim (Y)=\dim (X)=n.
\end{equation}
Hence, in this case any basis of the linear space
 $L$ over $k$ consists of the elements of the algebra $A$, which are algebraically independent over $k$, because, by construction, this algebra is generated by them; its transcendental degree over $k$ is then equal to the number of these elements:
\begin{equation*}
{\rm tr\,deg}_k(A)=\dim (Y)
\overset{\eqref{Yr}}{=\hskip -1mm=\hskip -1mm=}n
\overset{\eqref{c}}{=}\dim (L).
 \end{equation*}

 This proves that there is a $G$-equivariant isomorphism  $\alpha\colon Y\to
 L^*$, where $L^*$\,is a $kG$-module dual to the $kG$-module $L$.\;From \eqref{b} it follows that the action of $G$ on
 $L^*$ is the representtion $(\lambda^*)^*=\lambda$.\;By Theorem \ref{coli}
  there is a $G$-equivariant dominant morphism $\varepsilon\colon L^*\to
 L^*$, which is not a birational isomorphism.\;Therefore,
 the composition  $\varepsilon\circ\alpha\circ\varphi\colon X\to L^*$ is a nontrivial compression of the action $\varrho$ into the action $\lambda$.
\end{proof}

\begin{remark}\label{op}{\rm In view of \eqref{rdim},  as in the proof of statement (iii) we establish (under the assumption of affinity of $X$) the existence of

---\;\;an ${\rm rdim}_k(G)$-dimensional $kG$-submodule $M$ in $k[X]$, on which the action of the group  $G$ is faithful;

---\;\; an affine $G$-variety $Z$ and a dominant $G$-equivariant morphism $\psi\colon X\to Z$ such that $\psi^*(k[Z])$ is the $k$-subalgebra of $k[X]$ generated by the subspace  $M$.

Since the action of $G$  on $Z$ is faithful, $\psi$ is a compression of the action $\varrho$. If $n>{\rm rdim}_k(G)$, then $\psi$ reduces the dimension of $\varrho$, because $\dim (Z)\leqslant \dim_k(M)={\rm rdim}_k(G)$.

This gives another proof of the statement  (ii) of Theorem
\ref{nrdi}.}
\end{remark}

 \subsection{Compressing Abelian subgroups of rank \boldmath $2$ of the group ${\rm Cr}_2$}
Theorem \ref{rdim1av} answers the question about constructing finite Abelian subgroups of rank 1 of the Cremona group  ${\rm Cr}_n$  by means of base changes. For $n=2$, the next theorem answers the analogous question about the Abelian subgroups of rank  2,
i.e., the noncyclic subgroups isomorphic to  $\mathbb Z/a\mathbb Z\oplus  \mathbb Z/b\mathbb Z$,
$a\geqslant 2, b\geqslant 2$.

\begin{theorem}\label{rdim2} Let $\varrho$\,be a {\rm(}faithful\,{\rm)} rational action of a finite Abelian group $G$ of rank $2$
on $\bA2$
and let
$m_1,
m_2$\,be the sequence of invariant factors of the group $G$.
\begin{enumerate}[\hskip 2.2mm\rm(i)]
\item In every of the following cases
\begin{enumerate}[\hskip 4.0mm \rm(a)]
\item $|G|\neq 4$;
\item $|G|= 4$ and
$\varrho$
has a fixed point
\end{enumerate}
the rational action
 $\varrho$ is obtained by a nontrivial base change from a linear action
$\lambda\colon G\!\hookrightarrow\! \GL_2(k)$ on $\bA2$ such that
$\lambda(G)\!=\!
T_2(m_1, m_2)$
{\rm(}see\;\eqref{D}{\rm)}. In these cases the rational action $\varrho$ is incompressible into a rational action of a smaller dimension.
\item  If
$|G|= 4$ and $\varrho$ does not have a fixed point, then
 $G$ is a dihedral group $($that is a Klein's  Vierergruppe\hskip .3mm$)$,
 %%four-group
%% \hskip .3mm$)$,
and
$\varrho$ is obtained by a base change from the action
 $\gamma\colon G\to {\rm Cr_1}$
on ${\mathbf P}^1$, for which the group $\gamma(G)$
is generated by the elements
 $\sigma, \tau\in {\rm Aut}(\mathbf P^1)$ given by the formulas
\begin{equation}\label{st}
\mbox{
$\sigma\cdot (a_0:a_1)=(a_0:-a_1),\;\;\tau\cdot (a_0:a_1)=(a_1:a_0)$\;\;for all $(a_0:a_1)\in \mathbf P^1$}.
\end{equation}

%%\eqref{st}.
%%\begin{equation}\label{dh}
%%(x_0:x_1)\mapsto (x_0:-x_1)\quad\mbox{и}\quad (x_0:x_1)\mapsto (x_1:x_0).
%%\end{equation}
%%\item В
%%случаях {\rm (i)} и {\rm (ii)}
%%случае {\rm (i)} рациональное действие $\varrho$ несжимаемо в
%%    рациональное действие меньшей размерности.
\end{enumerate}
\end{theorem}

\begin{proof}
Since $G$ is an Abelian group of rank $2$, there exists a faithful linear represen\-ta\-tion
 \eqref{lrep} with $n=2$ and
$\lambda(G)=T_2(m_1, m_2)$.

If
$\varrho$ is compressible into a rational action
 $\vartheta$ of a smaller dimension, then $\vartheta$ is a faithful
 rational action of the group $G$ on an irreducible algebraic curve
 $C$. In view of the existence of a dominant rational map  $\bA2\dashrightarrow C$
 (a compression of $\varrho$ into $\vartheta$), the curve $C$
is rational.\;Therefore, we can (and shall) assume that $C=\mathbf P^1$
and hence
$G$ is isomorphic to a subgroup of  ${\rm Cr_1}={\rm Bir}(\mathbf P^1)={\rm Aut}(\mathbf P^1)={\rm PSL}_2$.\;It follows from the well-known description of finite subgroups in ${\rm PSL}_2$ that
noncyclic Abelian among them are only the subgroups conjugate
to the dihedral subgroup of order $4$, which is generated by the elements
$\sigma$ and $\tau$ given by formulas \eqref{st}.\;They do not have fixed points on $\mathbf P^1$.
%%So, in this case $G=D_2$.
In view of the ``going down'' property for fixed points (see\;\cite[Prop.\;A.2]{RY00_2}), it follows from
$(\mathbf P^1)^G=\varnothing $ that $\varrho$ does not have a fixed point.

If $\varrho$ is not compressible into a rational action
of smaller dimension, then by Theorem \ref{nrdi}(iii),
$\varrho$ is obtained by a nontrivial base change from a linear action
 $\lambda$ of the group $G$
on $\bA2$. From this and the
``going up'' property for fixed points
 (see\;\cite[Prop.\;A.4]{RY00_2}) it follows that if in the considered case
 both invariant factors $m_1$ and $m_2$ are equal to the same prime numer,
 then $\varrho$ has a fixed point. In particular, this is the case if $|G|=4$.
 This completes the proof of Theorem \ref{rdim2}.
\end{proof}

\subsection{Compressing other subgroups}
The classification of finite Abelian sub\-groups in ${\rm Cr}_2$ up to conjugacy is given in  \cite{Bl06}.\;In view of Theorems
 \ref{rdim1av} and \ref{rdim2}, it follows from it that among these subgroups only the subgroups
 %%that are
 isomorphic to
 %%are not investigated for non-trivial compressibility
 %% implies that among these subgroups
 %%not investigated for nontrivial compressibility there remain only subgroups isomorphic %%to
$$\mathbb Z/2d\mathbb Z\oplus  (\mathbb Z/2\mathbb Z)^2, d\geqslant 1;\;\; (\mathbb
Z/4\mathbb Z)^2\oplus  \mathbb Z/2\mathbb Z;\;\;
(\mathbb Z/3\mathbb Z)^3,\;\; \mbox{and}\;\; (\mathbb Z/2\mathbb Z)^4
$$
remain unexplored for nontrivial compressibility.\;Their ranks are 3, 3, 3, and 4, respectively. By Theorem \ref{nrdi}(i), all of these subgroups do not have fixed points.

\begin{theorem}\label{nal} Let $G$\,be a non-Abelian finite group
different from dihedral group and admitting a faithful linear representation $\lambda\colon G\hookrightarrow {\rm GL}_2(k)$.\;Then every {\rm(}faithful\,{\rm )}
rational action of the group $G$ on $\bA2$ is obtained by means of a nontrivial base change from its linear action  $\lambda$ on $\bA2$.
\end{theorem}

\begin{proof}
The statement will follow from Theorem \ref{nrdi}(iii)
if we prove that
$\varrho$ from this theorem cannot be compressed into a faithful
rational action of a smaller dimension.

Arguing on the contrary, assume that such a compression exists.\;Then, as in the proof of Theorem \ref{rdim2} we obtain that  $G$ is isomorphic to a subgroup of the group ${\rm Aut}(\mathbf P^1)={\rm PSL}_2$.\;Since noncyclic and nondihedral finite subgroups in ${\rm PSL}_2$ are only the rotation groups of a regular tetrahedron, octahedron, and icosahedron, $G$ is isomorphic to one of them. However, this is impossible because the rotation group of an icosahedron does not have nontrivial two-dimensional representations,
 and even though the rotation groups of  a regular tetrahedron and octahedron  have them, the kernels of these representations are nontrivial
 (their orders are 4), see, e.g.,  \cite{Vi85}.
Contradiction.
\end{proof}

 \section{Group embeddings and the Cremona groups}\label{2}
In this section, the characteristic  $k$ is zero.

\subsection{Properties of abstract Jordan groups}\label{2.1}
We recall the concepts introduced in  \cite[Def.\;2.1]{Po11}, \cite[Def.\;1]{Po14_1}.

For any finite group $H$, put
\begin{equation}\label{mF}
m_H:=\underset{S}{\rm min}\, [H:S],
\end{equation}
where $S$
runs over all normal Abelian subgroups of the group $H$.

\begin{definition}\label{J} Let $G$ be a group and let
\begin{equation}\label{Jc1}
J_G:=\underset{F}{\rm sup}\;m_F
\end{equation}
where $F$ runs over all finite subgroups of  $G$.
If $J_G<\infty$, then $G$ is called a {\it Jordan group} (one also says that $G$ has the {\it Jordan property}), and
 $J_G$ its {\it Jordan constant}.
\end{definition}

%%%%%%%%%%%%%%%
\begin{lemma} \label{Jsss}
For any groups $G_1, \ldots, G_s$, the inequality
\begin{equation}\label{JJ}
J^{\ }_{G_1\times\cdots\times G_s}\geqslant J_{G_1}\cdots J_{G_s}
\end{equation}
holds {\rm(}if $J_{G_i}=\infty$, then by definition, {\rm \eqref{JJ}} means that $J^{\ }_{G_1\times\cdots\times G_s}=\infty)$.
\end{lemma}

\begin{proof} Let $F_i$ be a finite subgroup of  $G_i$ and let $N$ be a normal Abelian subgroup of the finite subgroup $F_1\times\cdots\times F_s$ of the group $G_1\times\cdots\times G_s$.\;Let $\pi_i\colon F_1\times\cdots\times F_s\to F_i$ be the projection to the $i$th factor.\;Then $\pi_i(N)$ is a normal Abelian subgroup of the group $F_i$, therefore, \eqref{mF} implies the inequality
\begin{equation}\label{divis1}
|\pi_i(N)|\leqslant \frac{|F_i|}{\,m_{F_i}}.
\end{equation}
From the inclusion $N\subseteq \pi_1(N)\times\cdots\times\pi_s(N)$ and the inequality \eqref{divis1}, we get:
\begin{equation}\label{divis2}
|N|\leqslant |\pi_1(N)\times\cdots\times\pi_s(N)|=\prod_{i=1}^{s}|\pi_i(N)|\leqslant \prod_{i=1}^{s}\frac{|F_i|}{\,m_{F_i}}=\frac{|F_1\times\cdots\times F_s|}{m_{F_1}\cdots m_{F_s}}.
\end{equation}
It follows from \eqref{divis2} that $[(F_1\times\cdots\times F_s):N]\geqslant m_{F_1}\cdots m_{F_s}$, whence, in view of  \eqref{mF}, we obtain the inequality
\begin{equation}\label{divis3}
m^{\ }_{F_1\times\cdots\times F_s}\geqslant m_{F_1}\cdots m_{F_s}.
\end{equation}
Now \eqref{JJ} follows from \eqref{divis3} and \eqref{Jc1}.
\end{proof}
%%%%%%%%%%%%%%
%%%%%%%%%%%%

%%\begin{lemma}\label{Jsss} For every group $G$ and every integer $s>0$,
%%the inequality
%%\begin{equation}\label{JJ}
%%J^{\ }_{G^s}\geqslant J_G^{\hskip .2mm s}
%%\end{equation}
%%holds {\rm(}if $J_G=\infty$, by definition,  it means that $J^{\ }_{G^s}=\infty)$.
%%\end{lemma}

%%\begin{proof} Let $F$ be a finite subgroup of the group $G$, and let $N$ be a normal %%Abelian subgroup of the finite subgroup $F^s$ of the group $G^s$.\;Let $\pi_i\colon %%F^s\to F$ be the projection to the $i$th factor. Then $\pi_i(N)$ is a normal Abelian %%subgroup of the group $F$, therefore, \eqref{mF} implies the inequality
%%\begin{equation}\label{divis1}
%%|\pi_i(N)|\leqslant \frac{|F|}{\,m_F}.
%%\end{equation}
%%From the inclusion $N\subseteq \pi_1(N)\times\cdots\times\pi_s(N)$ and the inequality  %%\eqref{divis1}, we get:
%%\begin{equation}\label{divis2}
%%|N|\leqslant |\pi_1(N)\times\cdots\times\pi_s(N)|=\prod_{i=1}^{s}|\pi_i(N)|\leqslant %%\bigg(\frac{|F|}{\,m_F}\bigg)^{\hskip -1.3mm {s}}=\frac{|F^s|}{\,m_F^s}.
%%\end{equation}
%%It follows from \eqref{divis2} that $[F^s:N]\geqslant m_F^s$, whence, in view of %%\eqref{mF},
%%we obtain the inequality
%%\begin{equation}\label{divis3}
%%m^{\ }_{F^s}\geqslant m_F^s.
%%\end{equation}
%%Now \eqref{JJ} follows from \eqref{divis3} and \eqref{Jc1}.
%%\end{proof}

\begin{remark} {\rm
We set
$j_G:=\underset{F}{\rm sup}\;\underset{A}{\rm min}\, [F:A],
$
where $F$ runs over all finite subgroups of $G$, and $A$ over all Abelian (not necessarily normal) subgroups of $F$.\;Clearly $j_G\leqslant J_G$.\;The conditions  $J_G<\infty$ and $j_G<\infty$ turn out to be
equivalent,
see \cite[Rem.\;2.2]{Po11}.
Omitting the assumption of the normality of the subgroup $N$ in the proof of Lemma \ref{Jsss}, we obtain for any groups $G_1,\ldots, G_s$
the proof of the inequality
\begin{equation*}\label{JJJ}
j^{\ }_{G_1\times\cdots\times G_s}\geqslant j_{G_1}\cdots j_{G_s}.
\end{equation*}
 }
\end{remark}

\begin{theorem}\label{10} Let $P$ be a Jordan group and let $Q_1, \ldots , Q_s$ be the groups, each of which contains a non-Abelian finite subgroup.\;Then the group
$Q_1\times\cdots \times Q_s$ is nonembeddable in the group $P$ if $s>{\rm log}_2(J_P)$.
\end{theorem}
\begin{proof} It follows from Definition \ref{J} that $J_P< \infty$ and, if $Q_1\times\cdots \times Q_s$ is embeddable in $P$, then $J_{Q_1\times\cdots \times Q_s}\leqslant J_P$.\;This and Lemma \ref{Jsss} yield the inequality $J_{Q_1}\cdots J_{Q_s}\leqslant J_P$.\;But from \eqref{mF}, \eqref{Jc1}, and the condition on $Q_i$
it follows that $J_{Q_i}\geqslant 2$ for every $i$. Hence $2^s\leqslant J_P$, and therefore, $s\leqslant {\rm log}_2(J_P)$.
\end{proof}

\begin{remark}\label{rema}
{\rm
The statement and the proof of Theorem
\ref{10} remain in effect, if  in them $J_P$ is replaced by $j_P$, and $J_{Q_i}$ by $j_{Q_i}$.
}
\end{remark}

\subsection{Subgroups of the form \boldmath$G_1\times\cdots\times G_s$ and $p$-subgroups
in the Cremona groups}
We now apply the results from  Subsection \ref{2.1} to the Cremona groups.

\begin{theorem}\label{2121} Let $X$ be a rationally connected  variety $X$ defined over $k$.\;Then there exists an integer $b_X$, depending on $X$, such that every product of groups
$G_1\times\cdots\times G_s$, each of which contains a finite non-Abelian subgroup,
is nonembeddable in the group ${\rm Bir}_k (X)$ if $s>b_X$.
\end{theorem}

\begin{proof} This follows from Theorem \ref{10} and the Jordan property of the group ${\rm Bir}_k (X)$ (see the footnote in the Introduction).
\end{proof}

\begin{corollary}\label{sl} Let $n$ be a positive integer.\;Then there exists an integer
$b_{n, k}$, depend\-ing on $n$  and the field $k$,  such that every product of groups
 $G_1\times\cdots\times G_s$, each of which contains a finite non-Abelian subgroup,
 is nonembeddable in the Cremona group
 ${\rm Cr}_n(k)$ if $s>b_{n, k}$.
\end{corollary}
\begin{proof} This follows from Theorem \ref{2121} in view of rational connectedness of rational varieties.
\end{proof}

\begin{remark}{\rm
According to Theorem \ref{10} and Remark \ref{rema}, we can take $b_X={\rm log}_2(j^{\ }_{{\rm Bir}_k(X)})$
%%as $b_X$
in Theo\-rem\;\ref{2121}.\;The explicit upper bounds on $j^{\ }_{{\rm Cr}_2(k)}$ and $j^{\ }_{{\rm Bir}_k(X)}$ for
rationally connected threefolds $X$, as well as their exact values under certain rest\-ric\-ti\-ons, are found in  \cite{Se09}, \cite{PS17}.\;For example, if $k=\overline k$, then $j_{{\rm Cr}_n}=288$ and 10368 respectively for
$n=2$ and 3.
}
\end{remark}

\begin{corollary} For every prime integer $p$ and rationally connected  variety $X$ defined over $k$, there exists a non-Abelian finite $p$-group nonembeddable in ${\rm Bir}_k (\hskip -.3mmX\hskip -.3mm)$.\;In parti\-cu\-lar, for every integer $n>0$, there exists a non-Abelian finite  $p$-group nonem\-beddable in the Cremona group  ${\rm Cr}_n(k)$.
\end{corollary}

\begin{proof} This follows from Theorem  \ref{2121}, its Corollary  \ref{sl},
and the existence of finite non-Abelian $p$-groups.
\end{proof}

\subsection{Applications:\;\boldmath $p$-rank and embeddings of groups} Considering $p$-sub\-gro\-ups yields an obstacle to the existence of embeddings of groups.\;From here some applications are obtained.

Namely, let $p$ be a prime integer.\;Recall that a finite $p$-group is called {\it elementary} if it is Abelian and all its invariant factors (see above Subsection \ref{fabe}) are equal to\;$p$.

 \begin{definition}  For any group $G$ and prime integer  $p$, we call the {\it $p$-rank of the group  $G$} and denote by  ${\rm rk}_p(G)$ the least upper bound of ranks of all elementary  $p$-subgroups of the group $G$.
 \end{definition}

  \begin{example}\label{epr} {\rm Let the group $T$ be an $n$-dimensional torus  in the category of either affine algebraic groups over $\overline k$ or real Lie groups  (i.e.,
 $T$ is isomorphic to the product of $n$ copies of the group  ${\overline k}^\times$ в in the first case and the group $\{z\in \mathbf C^\times\mid |z|=1\}$ in the second case).
 It is not difficult to see that then
 ${\rm rk}_p(T)=n$
 for every prime integer\;$p$.\quad $\square$}
 \end{example}

 Clearly, if  $G_1$ and $G_2$ are two groups and ${\rm rk}_p(G_1)>{\rm rk}_p(G_2)$ for some $p$, then $G_1$ is nonembeddable in $G_2$.\;The applications of this remark are based on the fact that in some cases ${\rm rk}_p(G)$ can be explicitly computed or estimated. In particular, this is so for the Cremona groups:

\begin{theorem}\label{Rd} For any integer $n>0$, there exists a constant $R_n$,
depending on  $n$, such that ${\rm rk}_p(H)=n$ for every
 {\rm(}not necessarily closed\hskip .5mm{\rm)} subgroup $H$ of
 ${\rm Cr}_n$, containing an $n$-dimensional algebraic torus, and every
   $p>R_n$. In partucular,
 \begin{equation*}\label{rhod}
 {\rm rk}_p ({\rm Cr}_n)={\rm rk}_p ({\rm Aut}({\bf A}^{\hskip -.6mm n}))=n\;\;\mbox{for every $p >\!R_n$.}
 \end{equation*}
 \end{theorem}

\begin{proof} Let $p$ be a prime integer.\;It follows from Example \ref{epr} and the condition on $H$ that ${\rm rk}_p(H)\geqslant n$.\;On the other hand, combining
\cite[Thm.\;1.10]{PS16} with \cite[Cor.\;1.3]{Bi17}, we conclude that if $p$ is bigger than a certain constant $R_n$, depending on  $n$, then
${\rm rk}_p({\rm Cr}_n)\leqslant n$, and therefore, ${\rm rk}_p(H)\leqslant n$.
This completes the proof.
\end{proof}

The other examples of groups $G$, for which one manages to compute their $p$-rank, are
connected affine algebraic groups over $\overline k$ and connected real Lie groups.\;All the maximal tori in such a $G$ are conjugate;
$$\mbox{let $r(G)$ be the dimension of maximal tori of the group $G$}.$$
Recall that a prime integer $p$ is called  {\it a torsion prime of the group $G$} if $G$ has a finite Abelian $p$-subgroup not lying in any maximal torus.\;The torsion primes of the group $G$ divide the order of its Weyl group, so the set of all such primes is finite.

\begin{theorem}\label{al} Let $G$ be either a connected affine algebraic group
over $\overline k$ or a connected real Lie group.\;Suppose that a prime integer
$p$ is not a torsion prime of the group $G$.\;Then
${\rm rk}_p(H)=r(G)$ for any {\rm(}not necessarily closed\,{\rm)} subgroup  $H$ of $G$, containing a maximal torus of the group $G$.
\end{theorem}

\begin{proof} It follows from Example \ref{epr} and the condition on $H$ that
${\rm rk}_p(H)\geqslant r(G)$.
On the other hand, if
$F$ is a finite elementary  $p$-subgroup of  $H$, then
 $F$ lies in a maximal torus of the group $G$ because
 $p$ is not a torsion prime of $G$. In view of Example
\ref{epr}, this implies that the rank of  $F$ does not exceed  $r(G)$.\;Therefore, ${\rm rk}_p(H)\leqslant r(G)$.\;This completes the proof.
\end{proof}

From  Theorems \ref{Rd} and \ref{al} we obtain:

 \begin{corollary}\label{c1}
   Let $k$ be an algebraically closed field of characteristic zero.
   We associate with each positive integer $d$ any {\rm(}not necessarily closed\,{\rm)} subgroup $H_d$ from the following list:
 \begin{enumerate}[\hskip 4.2mm\rm(1)]
 \item a subgroup of the group ${\rm Cr}_d(k)$, containing a $d$-dimensional algebraic torus;
 \item a subgroup of any connected affine algebraic group  $G$ over $k$ with $r(G)=d$, containing a maximal torus of the goup $G$;
     \item a subgroup of any connected real Lie group Ли $G$ with $r(G)=d$,
 containing a maximal torus of the goup $G$.
 \end{enumerate}
 Then the group $H_n$ is nonembeddable in the group  $H_m$ if $n>m$. In particular,
 the following properties are equivalent:
  \begin{enumerate}[\hskip 4.2mm\rm(a)]
 \item the groups $H_n$ and $H_m$  are isomorphic;
 \item $n= m$.
 \end{enumerate}
\end{corollary}

Let us single out three particular cases as Corollaries \ref{sl1}, \ref{sl2} и \ref{sl3}.

\begin{corollary}[{{\rm  \cite[Thm.\;B]{Ca14}, \cite[Rem.\;1.11]{PS16}}}]\label{sl1}
The group ${\rm Cr}_n$ is embeddable in the group ${\rm Cr}_m$
if and only if  $n\leqslant m$.\;In particular, the groups ${\rm Cr}_n$ and
${\rm Cr}_m$ are isomorphic if and only if  $n=m$.
\end{corollary}

\begin{corollary}\label{sl2}  The group ${\rm Aut}_{\overline k}^{\hskip -.3mm n}$ is embeddable in the group ${\rm Aut}_{\overline k}^{\hskip -.3mm m}$
if and only if  $n\leqslant m$.\;In particular, the groups ${\rm Aut}_{\overline k}^{\hskip -.3mm n}$ and
${\rm Aut}_{\overline k}^{\hskip -.3mm m}$ are isomorphic if and only if  $n=m$.
\end{corollary}

Let $K_0:=k$ and $K_i:=k(x_1,\ldots, x_i)$ for $i=1,\ldots, n$.\;For any
 $a_i, b_i\in K_{i-1}$, $a_i\neq 0$,
   the map
   \eqref{string}, where $\sigma_i=a_ix_i+b_i$ for every $i$, is an element of the group ${\rm Cr}_n(k)$, and the set ${\rm B}_n(k)$ of all elements is a subgroup of  ${\rm Cr}_n(k)$.\;According to \cite{Po17}, the group ${\rm B}_n:={\rm B}_n({\overline k})$ is a Borel subgroup of ${\rm Cr}_n$; it contains the $n$-dimensional diagonal torus of the group ${\rm GL}_n$.

   \begin{corollary}\label{sl3}  The group ${\rm B}_n$ is embeddable in the group ${\rm B}_m$ if and only if
 $n\leqslant m$.\;In particular, the groups ${\rm B}_n$ and
${\rm B}_m$ are isomorphic if and only if  $n=m$.
\end{corollary}

As another application, we get

\begin{corollary}\label{conti} If $\varphi\colon {\rm Cr}_n\to {\rm Cr}_m$ is a continuous epimorphism of groups endowed with the Zariski topology, then  $n=m$ and $\varphi$ is an automorphism.
\end{corollary}

\begin{proof} In view of the topological simplicity of the group ${\rm Cr}_n$ (see \cite[Thm1.\;]{BZ18}), the kernel of the epimorphism $\varphi$ is trivial,
 and therefore it is an isomorphism of abstract groups. The statement
  now follows from Corollary
  \ref{c1}.
\end{proof}

\begin{remark}{\rm
According to \cite{Ur18}, it follows from \cite{BLZ18} the existence of an abstract group epimorphism
${\rm Cr}_3\to {\rm Cr}_2$.\;This shows that the assumption of conti\-nuity in
Corollary \ref{conti} is essential.\;On the other hand,
 ${\rm Cr}_2$ is a Hopfian abstract group, i.\,e., every its (not necessarily continuous)
 surjective endomorphism is an auto\-mor\-phism \cite{De07}.
}
\end{remark}

In the following theorem is used not the exact value of the $p$-rank of a group, but its upper bound.

\begin{theorem}\label{tp111} Let $M$ be a connected compact $n$-dimensional topological manifold and let $B_M$ be the sum of its Betti numbers with respect to homology with  coefficients in\;$\mathbf Z$.\;If
\begin{equation}\label{topineq}
d>\frac{\sqrt{n^2+4n(n+1)B_M}+n}{2}+{\rm log}_2B_M,
\end{equation}
then the Cremons group ${\rm Cr}_d$ is nonembeddable in the homeomorphism group  ${\mathscr H}(M)$ of the manifold $M$.
\end{theorem}

\begin{proof} Suppose that the inequality \eqref{topineq} holds.

Let $p\!>\!2$ be a prime integer satisfying the conditions:
\begin{enumerate}[\hskip 2.0mm\rm(i)]
\item $p>R_d$ (see Theorem \ref{Rd});
\item $p$ does not divide the order of the finite Abelian group
 $\bigoplus_{i=0}^n{\rm Tors}(H_i(M, \mathbf Z))$.
\end{enumerate}

It follows from \cite[Thm.\;2.5(3)]{MS63} that the rank of any finite elementary $p$-subgroup of the group
${\mathscr H}(M)$ does not exceed  $(\sqrt{n^2+4n(n+1)B_{M,p}}+n)/{2}+{\rm log}_2B_{M, p}$, where $B_{M,p}$ is the sum of Betti numbers of the manifold $M$
with respect to homology with coefficients in  ${\mathbf F}_p$.\;It follows from (ii)
and the universal coefficients theorem that
 $B_{M,p}=B_M$, whence, in view of \eqref{topineq},
we obtain the inequality
$d>{\rm rk}_p({\mathscr H}(M))$. From it, the condition (i), and  Theorem \ref{Rd}
we infer that ${\rm rk}_p({\rm Cr}_d)>{\rm rk}_p({\mathscr H}(M))$.
This completes the proof.
\end{proof}

According to \cite[Thm.\;1.10]{PS17}, \cite[Cor.\;1.3]{Bi17}, the constant  $R_d$ from Theorem \ref{Rd} can be chosen so that, for any  rationally connected
$d$-dimensional variety $X$ defined over $k$ and any prime integer $p>R_d$, the inequality  ${\rm rk}_p({\rm Bir}_k (X))\leqslant d$ holds.
From this, another statement about nonembeddable groups follows:

\begin{theorem}\label{t15} Let $X$ be a rationally connected $n$-dimensional variety $X$ defined over $k$, and let $p$ be a prime integer bigger than the constant $R_n$ from Theorem {\rm \ref{Rd}}\;Then any product of groups $G_1\times\cdots\times G_s$, each of which contains an element of order $p$, is nonembeddable
in the group ${\rm Bir}_k (X)$ if $s>d$.
\end{theorem}

\section{Connectedness of the Cremona groups}\label{1}
\subsection{A new proof of the connectedness theorem}

Two elements $\sigma$ and $\tau
\in \Cn(k)$ are called {\it linearly connected} if there exist a
$k$-defined open subset $U$ of the affine line $\mathbf A\!^1$ and a
$k$-morphism $\varphi\colon U\to \Cn$ such that $\sigma, \tau
\in \varphi(U(k))$.\;It is easy to verify that the relation of being linearly connected is an equivalence relation on $\Cn(k)$
 (see \cite[p.\;363]{Bl10}).\;By definition,  {\it linear connectedness} of the group $\Cn(k)$ means that there
is only one equivalence class of this equivalence relation.  Linear connectedness of the group $\Cn(k)$ implies its connectedness.

\begin{theorem}[{{\rm \cite{BZ18}}}]\label{con} The Cremona group $\Cn(k)$ is liearly connected if the field  $k$ is infinite.
\end{theorem}

\begin{proof}[Proof {\rm(}different from the proof in {\rm \cite{BZ18}}{\rm)}]
\

(a) First, ${\rm id}$ and every element $\sigma\in\Aff_n(k)$ are linearly connected,
because $\Aff_n$ is an open subset of the  $(n^2+n)$-dimensional affine space
${\mathscr A}_n$ of all affine maps  $\bAn\to \bAn$, and therefore, as $\varphi$ one can take
the identity map of the set  $U\!:=\!\ell\cap \Aff_n$, where $\ell$ is a line in ${\mathscr A}_n$, containing $\sigma$ and\;${\rm id}$.

(b) Second, every element $\sigma\in {\rm Bir}_k(\bAn)=\Cn(k)$
is of the form $\sigma=\alpha\circ\theta\circ\tau$, where  $\alpha, \tau\in \Aff_n(k)$, and
$\theta=(\theta_1,\ldots, \theta_n)\in \Cn(k)$  possesses the properties:
\begin{enumerate}[\hskip 2.8mm \rm(i)]
\item $\theta$ is defined at $o$;
\item $\theta(o)=o$;
\item $\theta$ is \'etale at
$o$, and
$d_o\theta
\colon  {\rm T}_{o,\bAn}\to {\rm T}_{o,\bAn}$
is the identity map.
\end{enumerate}

Indeed, since the map $\sigma\colon \bAn\dashrightarrow \bAn$ is
$k$-birational, and the field $k$ is infinite, there exists a point  $s\in\bAn(k)$,
at which
$\sigma$ is defined and \'etale  (its existence is equivalent to
the existence  of a point in $\bAn(k)$ that is not zero of some nonzero polynomial from
$k[x_1,\ldots, x_n]$).\;Now, as $\alpha$ and $\tau$ we can take any elements from $\Aff_n(k)$ such that $\tau^{-1}(o)=s$, $\alpha^{-1}(\sigma(s))=o$, and the composition of the maps
\begin{equation*}
{\rm T}_{o,\bAn}\xrightarrow{d_o\tau^{-1}}
{\rm T}_{s,\bAn}\xrightarrow{d_s\sigma}
{\rm T}_{\sigma(s),\bAn}\xrightarrow{d_{\sigma(s)}\alpha^{-1}}
{\rm T}_{o,\bAn}
\end{equation*}
is the identity map\,---\,obviously, such elements exist.

(c) We will now show that ${\rm id}$ and the element $\theta\in \Cr_n(k)$ specified in (b) are linearly connected. Clearly, in view of (a) and (b), this will complete the proof of Theorem \ref{con}.

Let ${\mathscr O}$ and $\widehat{\mathscr O}$\,be respectively the local ring of the variety $\bAn$ at the point  $o$ and its completion with respect to its maximal ideal.\;The set
of functions $x_1,\ldots, x_n$ is a system of local parameters of the variety
$\bAn$ at the point $o$.\;Therefore, we can (and shall) assume that  $\widehat{\mathscr O}={\overline k}[[x_1,\ldots, x_n]]$ and $\mathscr O$\,is the subring of $\widehat{\mathscr O}$ formed by the Taylor series at the point  $o$ of all the functions from $\mathscr O$ with respect to this system of local parameters.\;We have ${\mathscr O}_k:={\mathscr O}\cap k(\bAn)\subset
k[[x_1,\ldots, x_n]]$.

It follows from (i) that
$\theta_i\in {\mathscr O}_k$
for every $i=1,\ldots, n$, so we have
\begin{equation}\label{decomp}
  \theta_i=F_i(x_1,\ldots, x_n)\in k[[x_1,\ldots, x_n]]
\end{equation}
In view of (ii) and (iii), the series $F_i(x_1,\ldots, x_n)$ has the form
\begin{equation}\label{series}
F_i(x_1,\ldots, x_n)=
x_i+\sum_{d\geqslant 2}F_{i,d}(x_1,\ldots, x_n),
\end{equation}
where $F_{i,d}(x_1,\ldots, x_n)$ is a form of degree  $d$ in $x_1,\ldots, x_n$ with the coefficients in
  $k$, so we have
 \begin{equation}\label{homo}
 F_{i,d}(tx_1,\ldots, tx_n)=t^dF_{i,d}(x_1,\ldots, x_n)\quad \mbox{for any $t\in {\overline k}$.}
 \end{equation}
   From \eqref{decomp}, \eqref{series}, \eqref{homo} it follows that, for any $t\in \overline k$, the series
 $$tx_i+\sum_{d\geqslant 2}t^dF_{i,d}(x_1,\ldots, x_n)\!\in \widehat {\mathscr O}$$ lies in
 $\mathscr O$, and for
$t\in k$, it lies in
 $\mathscr O_k$.\;This implies that the series $$x_i+\sum_{d\geqslant 2}t^{d-1}F_{i,d}(x_1,\ldots, x_n)$$ also possesses the same properties.\;Therefore, for every
$t\in \overline k$, we obtain a rational map
 \begin{equation}\label{famil}
 \varrho(t)\colon \bAn\dashrightarrow\bAn, \quad \varrho(t)_i=
 x_i+\textstyle\sum_{d\geqslant 2}t^{d-1}F_{i,d}(x_1,\ldots, x_n),\quad\mbox{$i=1,\ldots, n$}.
 \end{equation}

In reality, $\varrho(t)\in \Cn$ for every $t$.\;Indeed,  \eqref{famil}
yields
\begin{equation}\label{id}
\varrho(0)=(x_1,\ldots, x_n)\overset{\eqref{string}}{=\hskip -1mm=}{\rm id}\in \Cn.
\end{equation}
 If $t\neq 0$ and
 $\vartheta(t):=(tx_1,\ldots, tx_n)\in \GL_n,$
 then from  \eqref{composi}, \eqref{decomp}, \eqref{series}, and \eqref{famil} we obtain
 \begin{equation}\label{conju}
 \vartheta(t^{-1})\circ\theta\circ \vartheta(t)=\varrho(t).
 \end{equation}
Since the left-hand side of the equality \eqref{conju} lies in $\Cn$, the same is true for the right one.

Thus, a mapping
$\varphi\colon \bAl
\to \Cn, t\mapsto \varrho(t)$,
arises.\;In view of \eqref{famil}, it is a $k$-morphism.\;From \eqref{id} and the equality  $\varrho(1)=\theta$ (following from \eqref{famil}, \eqref{series}, \eqref{decomp}) it now
follows that $\theta$ and ${\rm id}$ are
linearly connected.
\end{proof}

\subsection{The case of a finite field \boldmath$k$}

The following examples, belonging to A.\;Borisov \cite{Bo17}, show that
the condition of infinity of $k$  cannot be discarded in the given above proof.

\begin{examples}
Let $k
={\bf F}_q$, а $n=2$.\;Then the birational self-map
$\tau:=(x_1, x_2-1/(x_1^q-x_1))\in \Cp({\bf F}_q)$ is not defined at all points of ${\bf A}\!^2({\bf F}_q)$, and the birational self-map
 $\tau:=((x_1^q-x_1)x_1x_2, (x_1^q-x_1)x_2)\in \Cp({\bf
F}_q)$ is not \'etale at all such points.
\end{examples}

 \end{document}